\documentclass[12pt]{amsart}
\usepackage{sfilip_presets,amsmath,amsthm}

\usepackage[stretch=11]{microtype}
\newif\ifdebug
\debugfalse

\ifdebug
  \usepackage{lineno}
  \linenumbers
  \usepackage[status=draft,author=]{fixme}  
  \fxsetup{inline, marginclue,theme=color}
\else
  \usepackage[status=final, author=]{fixme}
\fi
\newcommand{\Wedge}{\bigwedge}

\usepackage{titletoc}

\let\oldtocsection=\tocsection
\let\oldtocsubsection=\tocsubsection
\let\oldtocsubsubsection=\tocsubsubsection

\renewcommand{\tocsection}[2]{\hspace{0em}\oldtocsection{#1}{#2}}
\renewcommand{\tocsubsection}[2]{\hspace{1.75em}\oldtocsubsection{#1}{#2}}
\renewcommand{\tocsubsubsection}[2]{\hspace{2em}\oldtocsubsubsection{#1}{#2}}
\usepackage{hyperref}
\usepackage{color}
\definecolor{darkred}{rgb}{0.5,0,0}
\definecolor{darkgreen}{rgb}{0,0.5,0}
\definecolor{darkblue}{rgb}{0,0,0.5}

\hypersetup{
    pdftitle={Semisimplicity and rigidity of the Kontsevich-Zorich cocycle},	% title
    pdfauthor={Simion Filip},	% author
    pdfsubject={math},		% subject of the document
    pdfkeywords={},	% list of keywords
    pdfnewwindow=true,		% links in new window
    colorlinks=true,		% false: boxed links; true: colored links
    linkcolor=darkblue,		% color of internal links (change box color with linkbordercolor)
    citecolor=darkred,		% color of links to bibliography
    filecolor=darkblue,		% color of file links
    urlcolor=darkblue,		% color of external links
    pdfborder={0 0 0},
    breaklinks=true
}
\makeatletter
\renewcommand{\paragraph}{%
\@startsection {paragraph}{4}
{\z@} \z@ {-\fontdimen 2\font }\bfseries
}
\makeatother
\newtheorem{theorem}{Theorem}[section]
\newtheorem{lemma}[theorem]{Lemma}
\newtheorem{proposition}[theorem]{Proposition}
\newtheorem{corollary}[theorem]{Corollary}

\theoremstyle{remark}
\newtheorem{remark}[theorem]{Remark}

\theoremstyle{definition}
\newtheorem{definition}[theorem]{Definition}

\numberwithin{equation}{subsection}
\title[Semisimplicity and rigidity of the KZ cocycle]{Semisimplicity and rigidity of the Kontsevich-Zorich cocycle}
\thanks{{Revised \textsc{\today}} }
\author{{ Simion Filip}\\
}

\address{
\parbox{0.5\textwidth}{
Department of Mathematics\\
University of Chicago\\
Chicago IL, 60615\\}
	}
\email{{sfilip@math.uchicago.edu}}
\begin{document}
%----------------------------
\begin{abstract}
We prove that invariant subbundles of the Kontsevich-Zorich cocycle respect the Hodge structure.
In particular, we establish a version of Deligne semisimplicity in this context.
This implies that invariant subbundles must vary polynomially on affine manifolds.
All results apply to tensor powers of the cocycle and this implies that the measurable and real-analytic algebraic hulls coincide.

We also prove that affine manifolds parametrize Jacobians with non-trivial endomorphisms.
Typically a factor has real multiplication.

The tools involve curvature properties of the Hodge bundles and estimates from random walks.
In the appendix, we explain how methods from ergodic theory imply some of the global consequences of Schmid's work on variations of Hodge structures.
We also derive the Kon\-tsevich-Forni formula using differential geometry.
\end{abstract}

\maketitle

%-----------TOC--------------
%\noindent\hrulefill
\tableofcontents
%\nointerlineskip
%\noindent \hrulefill
%----------------------------

%------Things to be fixed----
\ifdebug
  \listoffixmes
\fi
%----------------------------

\section{Introduction}

\subsection{Background}
\label{subsec:background}

On a Riemann surface $\Sigma$ giving a holomorphic $1$-form $\lambda$ is the same as giving charts where the transition maps are of the form $z\mapsto z\pm c$.
This datum is called a ``flat surface" and the group $\SL_2\bR$ naturally acts on it.
The action is on the charts and transition maps, after the identification of $\bC$ with $\bR^2$.
See the survey of Zorich \cite{Zorich_survey} for lots of context and motivation.

Flat surfaces with the same combinatorics of zeroes of the holomorphic $1$-form have a moduli space called a stratum and denoted $\cH(\kappa)$, where $\kappa$ is the multi-index encoding the zeroes.
The action of $\SL_2\bR$ preserves the Masur-Veech probability measure on such a stratum (\cite{Masur, Veech}) and one is interested in other possible invariant measures.

Recent work of Eskin and Mirzakhani in \cite{EM} shows that such measures must be of a very particular geometric form.
Further work by Eskin, Mirzakhani, and Mohammadi \cite{EMM} shows that these measures share properties with the homogeneous setting and unipotent actions.
In particular, all $\SL_2\bR$-orbit closures must be (affine invariant) manifolds.

To describe the local form of the measures, recall that on the stratum $\cH(\kappa)$ there are natural period coordinates (see \cite[Section 3.3]{Zorich_survey}).
Given a flat surface $\Sigma$ with zeroes of $\lambda$ denoted $S$, local period coordinates are given by the relative cohomology group $H^1(\Sigma,S;\bZ)\otimes \bR^2$.
The action of $\SL_2\bR$ is on the $\bR^2$ factor.

An \emph{affine invariant manifold} $\cM$ is an immersed closed submanifold of the stratum which in local period coordinates has an associated subspace $T_\cM\subset H^1(\Sigma,S;\bR)$.
The manifold $\cM$ must equal $T_\cM\otimes \bR^2$, and it then carries a natural invariant probability measure.
The results in \cite{EM} imply that any ergodic $\SL_2\bR$-invariant measure has to be of this form.

By work of McMullen (see for instance \cite{McMullen_Hilb_mod, McMullen_dynamics}) in genus~$2$ a much more detailed description is available.
Some of those results have also been independently obtained by Calta \cite{Calta}.

In the case of \Teichmuller curves (affine manifolds of minimal possible dimension), many results have been obtained by M\"oller (see for example \cite{Moller}).
They use techniques from variations of Hodge structure, but are on the algebro-geometric side.
In that context, for dimension reasons $\SL_2\bR$-invariant bundles (see below) are globally flat.
Moreover, the \Teichmuller curve is automatically algebraic and De\-ligne's semisimplicity results are available.

Part of this paper is concerned with extending the above results to affine manifolds.
The results from the global theory of variations of Hodge structures cannot be applied directly, because the structure at infinity of the affine manifolds is not clear.
This difficulty is bypassed using ergodic theory and the $\SL_2\bR$-action.
In fact, the methods in this paper are used in \cite{sfilip_algebraicity} to prove that affine invariant manifolds are quasi-projective varieties.

Our methods also provide an alternative route to some of the global consequences of Schmid's work \cite{Schmid}.
The main tools come from ergodic theory, rather than a local analysis of variations of Hodge structures on punctured discs.
The appendix contains a discussion of this application.

We also obtain rigidity results for $\SL_\bR$-invariant bundles.
In particular, any such measurable bundle has to be real-analytic.
This is used in the work of Chaika-Eskin \cite{Chaika_Eskin} on Oseledets regularity.

\subsection{Main results}

Recall that over a stratum we have the local system $E_\bZ$ corresponding to the absolute cohomology groups $H^1(\Sigma;\bZ)$ giving the Gauss-Manin connection.
The corresponding cocycle for the $\SL_2\bR$-action is called the Kontsevich-Zorich cocycle.
We also have the Hodge metric on the vector bundle $E_\bR:=E_\bZ\otimes_\bZ \bR$, which is not flat for the Gauss-Manin connection.

Our theorems concern $\SL_2\bR$-invariant subbundles of the bundle $E_\bR$ or $E_\bC$ (either of them denoted $E$).
To define these, fix a finite ergodic $\SL_2\bR$-invariant measure $\mu$.
An $\SL_2\bR$-invariant bundle is any measurable subbundle of $E$ which is invariant under parallel transport along a.e. $\SL_2\bR$-orbit.
It is defined $\mu$-a.e.

The results apply to subbundles of any tensor power of the Hodge bundle (still denoted $E$).

\begin{theorem}
Suppose $V\subset E$ is an $\SL_2\bR$-invariant subbundle.

Then $C\cdot V$ is also an $\SL_2\bR$-invariant subbundle, where $C$ denotes the Hodge-star operator.
In fact, $C$ can be any element of the Deligne torus $\bS$ (see \cite[Definition 1.4]{Deligne_travaux}).
\end{theorem}

With such a theorem available, we can prove the analogue of the Deligne semisimplicity theorem in this context.

\begin{theorem}[Deligne semisimplicity]
\label{thm:Deligne_ss_intro}
There exist $\SL_2\bR$-invariant bundles $V_i\subset E$ and vector spaces $W_i$, each equipped with Hodge structures and compatible actions of division algebras $A_i$, such that we have the isomorphism
 \begin{align}
 \label{eqn:ssimple_decomp}
 E\cong \bigoplus_i V_i\otimes_{A_i} W_i 
 \end{align}
 Moreover, the isomorphism is compatbile with the Hodge structures on the terms involved (in particular, the decomposition is Hodge-orthogonal).
 
 Any $\SL_2\bR$-invariant bundle $V'\subset E$ is of the form
 \[
 V' = \bigoplus_i V_i\otimes_{A_i} W_i'
 \]
 where $W_i'\subset W_i$ are $A_i$-submodules.
 See Remark \ref{remark:isotypical} for a discussion of how the isotypical components $W_i$ and symmetries $A_i$ can arise.
 
 For the case of the complexified bundle $E_\bC$ the subbundles $V_i$ have a Hodge structure in the following sense.
 They have components $V_i^{j,w-j}$ and the filtrations $V_i^{p}:=\oplus_{p\leq j} V_i^{j,w-j}$ vary holomorphically on \Teichmuller disks.
 The bundles $V_i$ are $\SL_2\bR$-invariant and carry a flat indefinite hermitian metric.
\end{theorem}

The theorem concerns $\SL_2\bR$-invariant bundles, which need not be flat in other directions.
In particular, the result is not implied by the usual Deligne semisimplicity (even assuming algebraicity of affine manifolds).

This applies, for instance, to the tautological bundle coming from the one-form giving the $\SL_2\bR$-action.
It is $\SL_2\bR$-invariant, but it is not flat unless the manifold is a \Teichmuller curve.
Results of Wright \cite{Wright} show that the projection of the tangent bundle to absolute homology has no flat subbundles.
As remarked, it always has the tautological $\SL_2\bR$-invariant subbundle.

Consider now flat (i.e. locally constant) subbundles on an affine manifold $\cM$.
By definition, a flat subbundle is one which is invariant under parallel transport along any path on the manifold.
It therefore corresponds to an invariant subspace in the monodromy representation.
The same kind of semisimplicity results as above hold here.

\begin{theorem}

\hfill

 \begin{description}
  \item [Fixed Part] Suppose $\phi$ is a flat section of the Hodge bundle (or any tensor power).
  Then $C\cdot \phi$ is also flat, where $C$ is the Hodge-star operator.
  This is equivalent to saying that each $(p,q)$-component of $\phi$ is also flat.
  
  \item [Semisimplicity] Suppose $V\subset E$ is an irreducible flat subbundle of the Hodge bundle (or any tensor power).
  Then so is $C\cdot V$.
  Moreover, the same kind of decomposition as in Theorem \ref{thm:Deligne_ss_intro} holds, but with flat subbundles instead of $\SL_2\bR$-invariant ones.
  In particular, the decomposition respects the Hodge structures.
 \end{description}
\end{theorem}

The decomposition into flat subbundles provided by the above theorem need not agree with the one from equation \eqref{eqn:ssimple_decomp} for $\SL_2\bR$-invariant ones.
The latter is a refinement of the former, however.

On affine invariant manifolds, $\SL_2\bR$-invariant bundles are even more rigid.
We prove that measurable subbundles have to depend, in fact, polynomially on the period coordinates.

\begin{theorem}
Suppose $V\subset E$ is a measurable $\SL_2\bR$-invariant subbundle on $\cM$.

Then it has a complement for which in local period coordinates on $\cM$ the operator of projection to it is polynomial (relative to a fixed flat basis).
\end{theorem}

A similar statement is proved in \cite{AEM} for the Forni subspace (see definition there).
They prove that it must be flat along the affine invariant manifold (i.e. locally constant).

The above results apply to the Hodge bundle and its tensor powers.
This implies the measurable and real-analytic algebraic hulls of the Kontsevich-Zorich cocycle have to agree.
For definitions see section \ref{sec:cocycles}.

The next result also applies to any tensor power of the cocycle.

\begin{theorem}
Suppose we have a measurable reduction of the Kontse\-vich-Zorich cocycle over some affine invariant manifold $\cM$.
This means that we have an algebraic subgroup $H\subset \Sp_g$ and a measurable choice of conjugacy class for $H$ in the automorphism group of each fiber.

Then in local period coordinates on $X$ this reduction must be real-analytic.
\end{theorem}

For example, given an invariant subbundle $V\subset E$, one can take at a point $x\in X$ its stabilizer $H_x:=Stab(V_x\subset E_x)$.

Another application, answering a question of Alex Wright, is that affine manifolds parametrize Riemann surfaces whose Jacobians have real multiplication.

\begin{theorem}
 Let $k(\cM)$ be the field of (affine) definition of the affine manifold $\cM$.
 It is defined in \cite[Theorem 1.5]{Wright}.
 
 Then this field is totally real and the Riemann surfaces parametrized by $\cM$ have Jacobians whose rational endomorphism ring contains $k(\cM)$.
 Moreover, the $1$-forms on $\cM$ giving the flat structure are eigenforms for the action of $k(\cM)$.
 
 In the case when $k(\cM)=\bQ$ the Jacobians have a non-trivial factor which contains the $1$-forms from $\cM$ (provided $\cM$ is not the entire stratum).
\end{theorem}

\subsection{Remarks and references} 

The question of invariant subbundles and their behavior has been extensively studied.
Starting with the work of Forni \cite{Forni}, the geometry of the Hodge metric has received a lot of attention.
The idea of studying the variation of Hodge structure which is available in this context goes back to Kontsevich \cite{Kontsevich}.
The work of M\"oller \cite{Moller} on \Teichmuller curves has introduced to the setting of flat surfaces many of the concepts used in this paper.

Our work is inspired in part by the questions raised by Forni, Matheus and Zorich \cite{FMZ_zero, FMZ}.
A central ingredient, reductivity of the algebraic hull, is from the paper \cite[Appendix A]{EM} (but follows directly from earlier results of Forni).
The expansion-contraction argument from section \ref{sec:rigidity} is standard, see for example \cite{AEM}.
The curvature calculations for the Hodge bundle are also standard, see for example \cite{Schmid}.

Throughout, we work in an appropriate finite cover of some connected component of a stratum.
This does not affect the statement or conclusions of any of the theorems, but has the advantage of avoiding orbifold issues.
In the appropriate finite cover, period coordinates exist locally and the Kontsevich-Zorich cocycle is well-defined.

\subsection{Outline of the paper}

In section \ref{sec:cocycles}, we prove a result about the image of the algebraic hull of a cocycle.
This is needed to extend the arguments to tensor powers of the Hodge bundle and in only necessary for the applications to the algebraic hull.
The result is likely known to experts .

In section \ref{sec:diff_geom} we collect standard properties of the Hodge bundles in a general variation of Hodge structure.
We compute the curvature of Hodge bundles, as well as a formula for the Laplacian of the log of the norm of a holomorphic section.
In favorable circumstances, the log of the norm is a subharmonic function.
This material is again classical and included because not all of it is readily available in the literature.

In section \ref{sec:random_walks} we prove some results about harmonic and subharmonic functions for random walks on groups.
To apply the standard techniques in variations of Hodge structure, we need control over such objects.

In section \ref{sec:ssimplicity} we assemble the developed material to prove the Theorem of the Fixed Part.
This is the first step, used then to deduce the semisimplicity result.
To streamline the arguments in this section we use the Deligne torus $\bS$ (see \cite{Deligne_travaux}).
This is the real-algebraic group with $\bR$-points equal to $\bC^{\times}$.
A Hodge structure on a real vector space is the same as a representation of this group.

In section \ref{sec:rigidity} we prove the theorem about polynomial dependence of invariant bundles.
First, the Hodge metric is used to prove real-analyticity along stable and unstable leaves.
This, augmented with an expansion-contraction argument, gives the polynomiality property.

In section \ref{sec:applications} we collect some applications.
First, we show that the measurable and analytic algebraic hulls of the cocycle have to coincide.
Then, we prove analogues of the semisimplicity theorem for flat bundles.
These are combined with results of Wright \cite{Wright} to show that the Jacobians over the affine invariant manifold admit real multiplication by the field of (affine) definition.

\paragraph{Acknowledgments}

I am very grateful to my advisor Alex Eskin for suggesting this circle of problems, as well as numerous encouragements and suggestions throughout the work.
His advice and help were invaluable at all stages.
I have also benefited a lot from conversations with Madhav Nori and Anton Zorich.
Julien Grivaux, Pascal Hubert, and Barak Weiss provided useful feedback on the exposition.

The questions about endomorphisms of Jacobians arose from conversations with Alex Wright.
I am grateful to him for discussions on this topic.

\section{Preliminaries on cocycles}
\label{sec:cocycles}

This section recalls the notions of cocycle, algebraic hull, and some of their properties.
These concepts are presented in more detail in the book of Zimmer \cite[Sections 4.2, 9.2]{Zimmer_book}. 
One proposition about the image of the algebraic hull is not available there (but likely known to experts) and for completeness is proved in this section.

\subsection{The setup and definitions}

Consider a standard Borel probability space $(X,\mu)$ equipped with an ergodic measure-preserving left action of a Polish group $A$.
Let also $G(k)$ denote the $k$-points of an algebraic group $G$, where $k$ is either $\bR$ or $\bC$.

\begin{definition}
A cocycle valued in $G(k)$ for the action of $A$ on $X$ is a map
\[
\alpha:A\times X \to G(k)
\]
satisfying the compatibility condition 
\[
\alpha(a_1 a_2,x)=\alpha(a_1,a_2x)\alpha(a_2,x)
\]
Two cocycles $\alpha$ and $\beta$ are said to be cohomologous (or conjugate) if there exists a function
$$
C:X\to G(k)
$$
such that
$$
\alpha(a,x)=C(ax)^{-1}\beta(a,x)C(x)
$$
\end{definition}
\begin{remark}
All maps in the definition are assumed measurable.

Note that the definition extends to the situation of bundles over $X$.
These can always be trivialized on a set of full measure, thus giving a cocycle as in the above definition.

Throughout this section, cocycles will be strict (in the sense of \cite[Section 4.2]{Zimmer_book}).
Whether certain identities hold a.e. or everywhere can be addressed as it is done in that section (and Appendix B, \emph{loc. cit.}). 
\end{remark}

Cohomologous cocycles have essentially equivalent dynamical properties, so one is interested in conjugating a given cocycle into a minimal subgroup of $G(k)$.

\begin{theorem}[{\cite[Prop. 9.2.1, Def. 9.2.2]{Zimmer_book}}]
With the setup as in the beginning of the section, there exists a  $k$-algebraic subgroup $L\subset G$ such that the cocycle $\alpha$ can be conjugated into $L(k)$, but cannot be conjugated into the $k$-points of a smaller $k$-subgroup of $G$.

The (equivalence class of) $L$ is called the \emph{algebraic hull} of $\alpha$.
\end{theorem}

\subsection{Algebraic hull under homomorphisms}

We consider the behavior of algebraic hulls under homomorphisms.

\begin{proposition}
Suppose $\alpha$ is a $G(k)$-valued cocycle and $\rho:G\to H$ is an algebraic representation. 
Then the algebraic hull of the cocycle $\rho\circ \alpha$ coincides with the image under $\rho$ of the algebraic hull of the cocycle $\alpha$.
\end{proposition}
\begin{proof}
Without loss of generality, we can assume that the algebraic hull of $\alpha$ is $G$ itself. 
Suppose, by contradiction, that the algebraic hull of $\rho\circ \alpha$ is a subgroup $F\subset H$ which is not $\rho(G)$.

First, we can assume that $F\subsetneq \rho(G)$. This follows from Zimmer's proof of the uniqueness (up to conjugation) of the algebraic hull. 
Indeed, ordering by inclusion subgroups into which the cocycle can be conjugated, he shows that any two minimal subgroups are conjugate. 
In particular, we can take a minimal element contained in $\rho(G)$.

Next, recall (\cite[4.2.18(b)]{Zimmer_book}) that reducing a cocycle to a subgroup $F(k)$ is the same as giving a $\rho\circ\alpha$ equivariant map
$$
\sigma:X\to H(k)/F(k)
$$
The action of $\rho(G)(k)$ on $H(k)/F(k)$ has locally closed orbits, in particular the quotient is a $T_0$ topological space and its Borel $\sigma$-algebra separates points.
Because the action of $A$ on $X$ is ergodic, we conclude the image of $\sigma$ must lie in a single $\rho(G)(k)$ orbit. One can therefore pick $t_0\in H(k)$ and a measurable section
$$
s:X\to \rho(G)(k)
$$
such that
$$
\sigma(x)=s(x)\cdot t_0 F(k)
$$
Writing out the equivariance condition for $\sigma$, which reads
$$
\sigma(ax)=\rho(\alpha(a,x)) \sigma(x)
$$
we obtain
$$
s(ax)t_0 F(k) = \rho(\alpha(a,x)) s(x) t_0 F(k)
$$
This implies
$$
t_0 F(k) = s(ax)^{-1} \cdot \rho(\alpha(a,x)) \cdot s(x)\cdot t_0 F(k)
$$
Multiplying on the right by $t_0^{-1}$, we deduce that
\begin{align}
\label{eqn:salphas}
s(ax)^{-1}\rho(\alpha(a,x)) s(x) \in t_0 F(k) t_0^{-1}
\end{align}
These elements also lie in $\rho(G)(k)$. Because $F\subsetneq \rho(G)$, it follows that $t_0 F t_0^{-1}\cap \rho(G)\subsetneq \rho(G)$. 
So, in fact, we could have reduced the cocycle to this latter subgroup, and could have done so using the coboundary
$$
s:X\to \rho(G)(k)
$$

We denote by $F$ the group $t_0 F t_0^{-1}$ and will show we could have conjugated the original cocycle into its preimage in $G$.
To do so, we must lift the map $s$ to $G(k)$.

By the remarks following Theorem 3.1.3 in \cite{Zimmer_book} the set $\rho(G(k))\backslash \rho(G)(k)$ is finite.
We rewrite equation \eqref{eqn:salphas} as 
$$
\rho(\alpha(a,x))s(x)\in s(ax)F(k)
$$
This implies the equality in the double coset space
$$
[s(ax)]=[s(x)] \textrm{ in } \rho(G(k))\setminus \rho(G)(k) / F(k)
$$
Because the action of $A$ is ergodic, this must land in a single double coset.
After choosing measurable sections for the corresponding actions, we find that
$$
s(x) = \rho(\tilde{s}(x))\cdot f(x)
$$
where $\tilde{s}(x)\in G(k)$ and $f(x)\in F(k)$.
In particular, we have
$$
\rho(\tilde{s}(x))^{-1} \rho(\alpha(a,x)) \rho(\tilde{s}(x)) \in F(k)
$$
We can thus use $\tilde{s}$ to make a change of basis for $\alpha$ to find that it lands in $\tilde{F}:= \rho^{-1}(F) \subsetneq G$.
This is a contradiction.
\end{proof}

\begin{corollary}
If the algebraic hull of a cocycle is reductive, then it stays reductive under any algebraic representation, in particular under considering various tensor operations.
\end{corollary}
\begin{proof}
This follows from the above proposition, since images of reductive groups stay reductive.
\end{proof}

\begin{remark}
\label{remark:reductive}
 We use \emph{reductive} to mean that any representation is semisimple, i.e. any invariant subspace has a complement.
 
 The Kontsevich-Zorich cocycle (for the $\SL_2\bR$-action) is reductive by the results in \cite[Appendix A]{EM} (see also Theorem 1.5 in \cite{AEM}), so the corollary applies to it.
 We can extend scalars from $\bR$ to $\bC$ and it will stay reductive.
\end{remark}

\section{Differential geometry of Hodge bundles}
\label{sec:diff_geom}

In this section we work with a holomorphic vector bundle $E$ equipped with a (pseudo-)hermitian metric denoted $\ip{-}{-}$.
The bundle is over some unspecified complex manifold.
Let $\nabla$ denote the corresponding Chern connection, uniquely defined by the two properties
\begin{align*}
 d\ip{\alpha}{\beta} & = \ip{\nabla\alpha}{\beta} + (-1)^{|\alpha|}\ip{\alpha}{\nabla \beta}\\
 \nabla&(\textrm{holomorphic}) \in \cA^{1,0} (E)
\end{align*}
Here $\cA^{1,0}(E)$ denotes forms of type $(1,0)$ on the complex manifold having as coefficients sections of the bundle $E$.

The Chern connection is explicit once we fix a local holomorhpic basis of $E$.
The connection then has the form $\nabla=d+A$ and the hermitian metric is given by the matrix $h$.
The relation between them becomes
$$
A=\partial h \cdot h^{-1}
$$

Throughout this section $\Omega$ denotes the curvature of the connection $\nabla$.
If we have a local trivialization $\nabla=d+A$ then $\Omega=dA+A\wedge A$ (see Remark \ref{remark:dA_AA} for the sign).

First, we provide a formula for what can be considered the second variation of the log of the norm of a holomorphic section.
Next, we prove formulas for the curvature of quotient and subbundles.
These are used to derive the formulas for the curvature of the Hodge bundles.
Finally, positivity for differential forms is recalled.

The material in this section is standard.
It is included here because different sources use different conventions and it is hard to refer to a single place.

\subsection{Second variation formula}

\begin{lemma}
\label{lemma:delbar_del_log}
 Suppose $\phi$ is a holomorphic section of $E$. Then we have the formula
\begin{align*}
 \delbar\del \log \norm{\phi}^2 = \frac{\ip{\Omega\phi}{\phi}}{\norm{\phi}^2} + \frac{\ip{\nabla\phi}{\phi}\ip{\phi}{\nabla\phi} - \norm{\phi}^2\ip{\nabla\phi}{\nabla\phi}}{\norm{\phi}^4}
\end{align*}
\end{lemma}
\begin{proof}
 We shall use throughout the fact that the connection respects the metric and that
 $$
 d\ip{\alpha}{\beta} = \del\ip{\alpha}{\beta} + \delbar\ip{\alpha}{\beta}
 $$
First, we have
$$
\partial \log \norm{\phi}^2 = \frac{1}{\norm\phi^2}\cdot \partial \norm \phi ^2 = \frac{\ip{\nabla \phi}{\phi}}{\norm \phi^2}
$$
Next, we apply the product rule
$$
\delbar\left(\partial \log \norm \phi ^2\right) = \left(\delbar \frac 1 {\norm \phi ^2}\right)\ip{\nabla\phi}{\phi} + \frac{1}{\norm \phi ^2}\delbar\ip{\nabla \phi}{\phi}
$$
For the first term, we have
$$
\delbar \frac 1 {\norm \phi^2} = \frac{-1}{\norm\phi ^4}\cdot \delbar\norm{\phi}^2 = \frac{-\ip{\phi}{\nabla \phi}}{\norm \phi^4}
$$
For the second term, we have
$$
\delbar \ip{\nabla \phi}{\phi} = (1,1)-\textrm{part of } d\ip{\nabla \phi}{\phi}
$$
However, $d\ip{\nabla\phi}{\phi}$ has only a $(1,1)$-part.
Indeed, its $(2,0)$-part can be identified with $\del\del\ip{\phi}{\phi}$ and this vanishes always.

By applying the product rule for the Chern connection, we then have
$$
d\ip{\nabla\phi}{\phi} = \ip{\Omega\phi}{\phi} - \ip{\nabla \phi}{\nabla \phi}
$$
Combining the above, we at last have
\begin{align*}
 \delbar\del \log \norm{\phi}^2 &= \frac{-\ip{\phi}{\nabla\phi}\cdot\ip{\nabla\phi}{\phi} - \norm{\phi}^2 \cdot \ip{\nabla\phi}{\nabla\phi}}{\norm{\phi}^4} + \frac{\ip{\Omega\phi}{\phi}}{\norm{\phi}^2}=\\
 &= \frac{\ip{\Omega\phi}{\phi}}{\norm{\phi}^2} + \frac{\ip{\nabla{\phi}}{\phi}\cdot\ip{\phi}{\nabla\phi} - \norm{\phi}^2\cdot \ip{\nabla \phi}{\nabla\phi}}{\norm{\phi}^4}
\end{align*}
Note that $\nabla\phi$ is a $(1,0)$-form (the reason for a sign switch above) and moreover, the second term is always a negative $(1,1)$-form (see Section \ref{subsec:positivity}) by the Cauchy-Schwartz inequality.
\end{proof}

\begin{remark}
\label{remark:delta_masses}
\leavevmode
\begin{itemize} 
 \item[(i)] In the case of Hodge bundles, the curvature term in the above calculation will be arranged to be negative as well, yielding a subharmonic function.
 
 \item[(ii)] In the situation when $\phi$ has zeroes, the above equation has to be interpreted.
 In the distributional sense, we get a current corresponding to the zero-divisor.
 This will not affect the discussion as we'll be concerned with subharmonic functions, and adding this current does not affect the conclusion.
 
 Indeed, our considerations will be in complex dimension one on the base.
 We have that $\log |z|$ is subharmonic in $\bC$, owing to the distributional identity
 $$
 \delbar \del \log |z| = 2\pi \sqrt{-1}\delta_0
 $$
 If we have a holomorphic section $\phi$ with zero of order $k$ at the origin, we can write it as $\phi=z^k v(z)$ where $v(z)$ is also a holomorphic section, now without any zeroes.
 We then have
 $$
 \log \norm{\phi}^2 = 2k \log |z| + \log \norm{v(z)}^2
 $$
 This implies that $\log\norm{\phi}^2$ is a subharmonic function, even near the zeroes of $\phi$.
\end{itemize}
\end{remark}

\begin{remark}
\label{remark:delbardel_norm}
The following formula is also useful (see Equation (7.13) in \cite{Schmid}):
 \begin{align}
 \label{eqn:delbardel_norm}
 \delbar \del \norm{\phi}^2 = \ip{\Omega \phi}{\phi} - \ip{\nabla \phi}{\nabla \phi}
 \end{align}
To prove it, note first that
$$
\del \norm{\phi}^2 = \ip{\nabla \phi}{\phi}
$$
Take now the differential of the above and consider its $(1,1)$-component. This gives
$$
\delbar\ip{\nabla \phi}{\phi} = \ip{\Omega \phi}{\phi} - \ip{\nabla \phi}{\nabla \phi}
$$
Suppose now that $\norm{\phi}^2$ is constant, and the curvature term is negative (see Section \ref{subsec:positivity} for conventions).
Then each of the terms in equation \eqref{eqn:delbardel_norm} must vanish, so we find
$$
\nabla \phi = 0 \hskip 1cm \textrm{and} \hskip 1cm \Omega \phi=0
$$
In the context of Hodge bundles, this will give that certain sections are flat and holomorphic (see Lemma 7.19 in \cite{Schmid} for the exact analogue).
\end{remark}

\subsection{Quotients and Subbundles}

Consider a short exact sequence of holomorphic vector bundles
$$
0\to S\into E\onto Q\to 0
$$
Suppose that $E$ is equipped with a non-degenerate hermitian form $\ip{-}{-}$, not necessarily positive-definite.
Assume however that its restriction to $S$ is non-degenerate.
That is, we assume that $S$ is disjoint from $S^\perp$, which we identify with $Q$ by the natural map.

\begin{remark}
The condition on the hermitian form not being positive definite is relaxed because it is not true in the case of the Hodge bundles.
The indefinite metric there plays a crucial role and corresponding signs are essential.
\end{remark}

The above decomposition gives a projection operator $P:E\to S$ and considering the map $(1-P)\circ \nabla$ composed with the projection, we have the second fundamental form
$$
B:S\to Q \otimes \cA^{1,0}
$$
Let $B^\dag:Q\to S\otimes \cA^{0,1}$ be the adjoint of $B$ with respect to the hermitian form.

\begin{proposition}
\label{prop:curv_sub_quot}
 The curvature of the bundles $S$ and $Q$ is given by
 \begin{align*}
  \Omega_S&=\left.\Omega_E\right|_S + B^\dag\cdot B\\
  \Omega_Q&=\left.\Omega_E\right|_Q + B\cdot B^\dag
 \end{align*}
 Here, $\left.\Omega_E\right|_?$ represents the restriction of the curvature of $E$ to the corresponding subbundle.
 \end{proposition}
\begin{proof}
 If we denote the connection matrix on $E$ by $A_E$, we have
 $$
 A_E=
 \begin{bmatrix}
  A_S & -B^\dag\\
  B & A_Q
 \end{bmatrix}
 $$
 where $A_S$ and $A_Q$ are the connection matrices for the bundles $S$ and $Q$.
 Recall now that the curvature is $dA + A\wedge A$, so we have
 $$
 \Omega_E=
 \begin{bmatrix}
  \Omega_S - B^\dag B & *\\
  * & \Omega_Q - B B^\dag
 \end{bmatrix}
 $$
 This yields the claimed formula.
\end{proof}

\begin{remark}
\label{remark:dA_AA}
 In various places in the literature, the formula for the curvature appears as either $dA-A\wedge A$ or $dA+A\wedge A$.
 The issue is that once a trivialization of the bundle is fixed, we can write $\nabla = d+A$, where $A$ is an operator.
 Then we have $\nabla^2 f = (dA)f - A(Af)$ (due to the sign rule), but if we write out $A$ as a matrix of $1$-forms, then
 $\nabla^2 f = (dA)f + (A\wedge A)f$. 
\end{remark}

\subsection{Positivity}
\label{subsec:positivity}

Below are the conventions about which forms are positive and which ones are negative.
\begin{definition}
 A purely imaginary $(1,1)$-form $\omega$ is \emph{positive} if it can be written 
 $$
 \omega = \sum h_{ij} dz^i\wedge \conj{dz^j}
 $$
 where $h_{ij}$ is a positive hermitian matrix.
 For example, $dz\wedge\conj{dz}$ is positive.
 
 A purely real $(1,1)$-form $\omega$ is \emph{positive} if $\sqrt{-1}\omega$ is a positive form.
 For example, $dx\wedge dy = \sqrt{-1} dz\wedge\conj{dz}$ is positive. 
 
 A form $\Omega\in \cA^{1,1}\otimes \End(E)$ is \emph{positive} if for any section $e\in \Gamma(E)$ we have that $\ip{\Omega e}{e}$ is positive. 
\end{definition}
\begin{remark}
 Note that an equivalent definition of positivity for imaginary $(1,1)$-forms is that for any tangent vector $\xi$ we have $\omega(\xi,\conj{\xi})\geq 0$.
\end{remark}

\begin{proposition}
 Consider the setting of the previous section, where curvatures of quotients and subbundles was computed.
 \emph{Assume that the metric is positive-definite.}
 Then we have that $BB^\dag$ is positive and $B^\dag B$ is negative.
\end{proposition}
\begin{proof}
 Because the bundles are holomorphic, we have $B\in \cA^{1,0}\otimes\Hom(S,Q)$. 
 This can be seen by choosing a holomorphic trivialization for $S$ and computing the matrix $B$.
 A change of frame will not affect the type of $B$.
 
 We have
 $$
 \ip{BB^\dag q}{q} = -\ip{B^\dag q}{B^\dag q}
 $$
 This is a \emph{positive} form (of type $-\conj{dz}\wedge{dz}$).
 Similarly
 $$
 \ip{B^\dag Bs}{s} = - \ip{Bs}{Bs}
 $$
 which is a \emph{negative} form.
\end{proof}
\begin{remark}
 The claim above concerning the positivity used the definiteness of the hermitian form.
 But the curvature calculation remains valid without this assumption.
\end{remark}

\subsection{Curvature of Hodge Bundles}

\paragraph{Setup}
Consider a variation of polarized Hodge structures of weight $w$ over some fixed complex manifold.
This is the data of a flat bundle $H_\bC$ equipped with the Gauss-Manin connection $\nabla^{GM}$.
We further have a filtration by holomorphic subbundles 
$$
\ldots\subset \cF^p \subset \cF^{p-1}\subset\ldots \subset H_\bC
$$
Denote the quotient subbundles by
$$
\cH^{p,q} := \cF^p/\cF^{p+1}
$$
The polarization provides the indefinite form $\ip{\cdot}{\cdot}_i$ which is flat for the Gauss-Manin connection.
By assumption, we also have the definite metric
$$
\ip{\cdot}{\cdot} := \ip{C\cdot}{{\cdot}}_i
$$
Here $C$ is the Weil, i.e. Hodge-star, operator (note the indefinite metric already has a conjugation in the definition).
We also view $\cH^{p,q}$ as subbundles of $H_\bC$, but note that they are not holomorphically embedded (for the holomorphic structure coming from the Gauss-Manin connection).
Restricted to $\cH^{p,q}$, the definite and indefinite metrics agree up to a sign.

Note that $\nabla^{GM}$ is the Chern connection on $H_\bC$ equipped with the indefinite metric and complex structure coming from the flat structure.
Viewing $H_\bC$ as the direct sum of the holomorphic bundles $\cH^{p,q}$, each equipped with the definite metric, we also have the Hodge connection $\nabla^{Hg}$.
It is defined as the Chern connection of $\oplus \cH^{p,q}$ equipped with the definite metric (and taking direct sums).

\begin{remark}
\label{remark:holomorphic_struct}
The bundle $H_\bC$ carries two different complex structures and metrics.
On the one hand, we have the flat structure (inducing a holomorphic one) and indefinite metric.
On the other, we have a direct sum of holomorphic bundles, the $\cH^{p,q}$, each equipped with a definite metric.
\end{remark}

Consider also the second fundamental form (for the indefinite metric)
$$
\sigma_p: \cH^{p,q}\to \cH^{p-1,q+1}
$$
Note that it is at first defined as $\sigma_p:\cF^p\to H_\bC/\cF^p$, but the Griffiths transversality condition implies it must in fact map subspaces as above.

Finally, let $\sigma^\dag_p$ denote the adjoint of $\sigma_p$ for the indefinite metric. 
It differs from the adjoint for the definite metric by exactly one minus sign.
We then have the equality of connections
\begin{align}
\label{eqn:GM_vs_Hg}
\nabla^{GM}=\nabla^{Hg} + \sigma_\bullet + \sigma^\dag_\bullet
\end{align}
Note that the curvature of $\cH^{p,q}$ for either metric is the same, since they agree up to a sign.

\begin{proposition}
\label{prop:curvature_Hg}
 We have the formula for the curvature
 $$
 \Omega_{\cH^{p,q}}= \sigma^\dag_p\wedge\sigma_p + \sigma_{p+1}\wedge \sigma^\dag_{p+1}
 $$
\end{proposition}
\begin{proof}
 By the remark above, it suffices to compute the curvature for the indefinite metric.
 From the exact sequence of bundles
 $$
 0\to \cF^p\into H_\bC\onto H_\bC/\cF^{p}\to 0
 $$
 we find using Proposition \ref{prop:curv_sub_quot} that
 $$
 \Omega_{\cF^p} = \sigma^\dag_p\wedge\sigma_p
 $$  
 Next, consider the exact sequence
 $$
 0\to \cF^{p+1}\into \cF^p\onto \cH^{p,q}\to 0
 $$
 Again Proposition \ref{prop:curv_sub_quot} yields
 \begin{align*}
  \Omega_{\cH^{p,q}} &= \Omega_{\cF^p} + \sigma_{p+1}\wedge \sigma_{p+1}^\dag \\
  &= \sigma^\dag_p\wedge \sigma_p + \sigma_{p+1}\wedge \sigma_{p+1}^\dag  
 \end{align*}
 This is the claimed formula. 
\end{proof}
\begin{remark}
 This formula agrees with that in Lemma 7.18 of \cite{Schmid}.
 Note that in \emph{loc. cit.} adjoints are for the \emph{definite} metric, so formulas differ by a minus sign everywhere.
\end{remark}

For future use, we also record the following result.
\begin{proposition}
\label{prop:curvature_Hg_sect}
 Suppose $e,e'$ are two smooth sections of $\cH^{p,q}$.
 Then for the definite metric, we have the formula
 \begin{align}
 \label{eqn:curvature_ip}
 \ip{\Omega_{\cH^{p,q}}e}{e'} = \ip{\sigma_p e}{\sigma_pe'}+\ip{\sigma_{p+1}^\dag e}{\sigma_{p+1}^\dag e'}
 \end{align}
\end{proposition}
\begin{proof}
 This will follow from the fact that on $\cH^{p,q}$, we have 
 $$\ip{-}{-}_i=(-1)^p\ip{-}{-}$$
 Note that whenever we exchange two $1$-forms, a sign gets switched.
 We abbreviate $\Omega_{\cH^{p,q}}$ by $\Omega$.
 \begin{align*}
  \ip{\Omega e}{e'}&=(-1)^p \ip{\Omega e}{e'}_i=\\
  &=(-1)^p\left( \ip{\sigma_p^\dag\wedge \sigma_p e}{e'}_i + \ip{\sigma_{p+1}\wedge\sigma^\dag_{p+1}e}{e'}_i \right)\\
  &=(-1)^{p+1} \left(\ip{\sigma_p e}{\sigma_p e'}_i + \ip{\sigma^\dag_{p+1}e}{\sigma^\dag_{p+1} e'}_i \right)\\
  &=(-1)^{p+1} \left( (-1)^{p-1} \ip{\sigma_p e}{\sigma_p e'} + (-1)^{p+1} \ip{\sigma^\dag_{p+1}e}{\sigma^\dag_{p+1} e'}  \right)  
 \end{align*}
The desired formula then follows.
\end{proof}

\begin{corollary}
 The ``rightmost" bundle $\cH^{0,w}$ has negative curvature.
\end{corollary}
\begin{proof}
 The second fundamental form $\sigma_0$ vanishes in this case, so the only curvature term in equation \eqref{eqn:curvature_ip} involves $\sigma_1^\dag$.
 The corresponding term is negative-definite.
\end{proof}

\begin{remark}
 The above calculations are standard, and presented in detail for example in Section 7 of \cite{Schmid}.
 But in order to apply the same techniques as in Lemma 7.19 and Theorem 7.22 of \cite{Schmid}, one needs control over subharmonic functions on \Teichmuller disks.
 This is addressed in the next section.
\end{remark}

\section{Random Walks}
\label{sec:random_walks}

\paragraph{Setup} Suppose $G:=\SL_2\bR$ acts (on the left) on a measure space $X$, preserving a probability measure $\mu$. 
Let also $\nu$ be a measure on $G$ with compact support.
For this section, only $\nu$-stationarity of $\mu$ is required.

We also assume that the action of $G$ on $(X,\mu)$ is ergodic and that the support of $\nu$ generates $G$.
This suffices for the Furstenberg Random Ergodic Theorem to hold.
The survey of Furman \cite{Furman} (see Section 3) provides a discussion of the needed facts.

We shall need the following form of the Random Ergodic Theorem.

\begin{theorem}[Furstenberg]
With the setup as above, consider a function $f\in L^1(X,\mu)$.
Then for a.e. $(x,\omega)\in X\times G^{\bN}$ we have
$$
\lim_{N\to \infty} \frac 1 N \sum_{i=0}^{N-1} f(g_i(\omega)\cdots g_0(\omega) x) = \int_X f\, d\mu
$$
Moreover, suppose that $f:X\to \bR_{\geq 0}$ takes only positive values.
Then the same conclusion holds, even if the integral is $+\infty$.
\end{theorem}

The second statement is not usually part of the Random Ergodic Theorem, but clearly follows by applying the first part to the truncated above function.

\subsection{Harmonic functions}

\begin{definition}
\label{def:subharm}
For a measurable function $f:X\to \bR$ define
$$
(\nu*f) (x) :=\int_G f(gx) \, d\nu(g)
$$
The function $f:X\to \bR$ is said to be \emph{$\nu$-harmonic} if we have for a.e. $x\in X$
 $$
 f(x)=(\nu*f) (x)
 $$
 It is said to be \emph{$\nu$-subharmonic} if we have for a.e. $x\in X$
 $$
 f(x)\leq (\nu*f)(x)
 $$
 Part of the definition is that $\nu*f$ is well-defined.
 
 Define also the analogue of the \emph{Laplacian}
 $$
 Lf:=\nu*f-f
 $$
\end{definition}

Now, assume $G$ is endowed with some non-trivial norm $\norm{-}$ satisfying the triangle inequality. 
Assume it gives a left-invariant distance inducing the same topology. 
For $\SL_2\bR$ the operator or matrix norm will do.

\begin{definition}
 A measurable function $f:G\to \bR$ is \emph{tame} if it satisfies the bound 
 $$
 |f(g)|=O(\norm{g})
 $$
 A measurable function $f:X\to \bR$ is \emph{tame} if for $\mu$-a.e. $x\in X$, the function $f_x$ defined by
 $$
 f_x(g)=f(gx)
 $$
 is a tame function on $G$.
\end{definition}

The following proposition puts some restrictions on (sub)harmonic functions on $X$.

\begin{proposition}

\label{prop:pos_tame_harmonic}
\leavevmode
\begin{itemize}
 \item[(i)] Suppose that $f\in L^1(X,\mu)$ is $\nu$-subharmonic.
 Then $f$ is a.e. constant.
 
 \item[(ii)] Suppose that $f:X\to \bR$ is positive, tame, and $\nu$-harmonic.
 Then $f$ is a.e. constant.
\end{itemize}
 \end{proposition}
\begin{proof}
 Consider a random walk on $G$, sampled by the measure $\nu$.
 For part (i) note that by subharmonicity we have
 $$
 f(x) \leq \frac 1 N \sum_{1}^N \bE[f(g_n \cdots g_1 x)]
 $$
 By the Furstenberg Random Ergodic theorem, the right-hand side converges a.e. to $\int_X f\, d\mu$. 
 We thus have
 $$
 f(x)\leq \int_X f\, d\mu
 $$
 Integrating the above inequality over $X$ for the measure $\mu$, we see that equality must occur $\mu$-a.e.
 
 For (ii) note that the tameness of $f$ implies that $\nu*f$ is well-defined and still tame.
 Now, because $f$ is harmonic, we have
 $$
 f(x) = \frac 1 N \sum_1^N \bE[f(g_n\cdots g_1 x)]
 $$
 By the Furstenberg Random Ergodic theorem, the right hand side converges to $\int_X f\, d\mu$.
 If the integral is finite, we conclude as before.
 
 If this integral is $+\infty$, then $f$ must also be infinite a.e.
 This contradicts the tameness of $f$. 
\end{proof}

\subsection{Subharmonic functions with sublinear growth}

We keep the setup from the previous section.

\begin{definition}
 A function $f:X\to \bR$ is of \emph{sublinear growth} if for a random walk sampled from $\nu$ we have for $\mu$-a.e. $x$ that $\omega$-almost surely
 $$
 |f(g_n(\omega)\cdots g_1(\omega) x)|=o(n)
 $$
 The estimate is allowed to depend on $x$ and $\omega$.
 Here, $\omega$ denotes the point in the (unspecified) probability space modeling the random walk.
\end{definition}
\begin{remark}
 It is possible for a function to be of sublinear growth, yet not be tame.
\end{remark}

\begin{proposition}
\label{prop:subharm_const}
\leavevmode
\begin{itemize}
 \item[(i)] Suppose that $f$ is positive, tame, of sublinear growth, and $\nu$-subharmonic.
  Then $f$ is constant.
  \item[(ii)] Suppose that $f$ is $\nu$-subharmonic.
Let $f^+:=\max(0,f)$ be its positive part.
Assume $f^+$ is tame and of sublinear growth (it automatically is $\nu$-subharmonic, as the $\max$ of two such).

  Then $f$ is constant.
\end{itemize}
\end{proposition}
\begin{proof}
 Part (ii) is of course stronger, but Part (i) is needed to deduce it.
 We prove it first.
 
 Consider $Lf:=\nu*f-f$.
 Because $f$ is subharmonic, this function is non-negative.
 We shall prove that it must be zero, thus reducing this statement to Proposition \ref{prop:pos_tame_harmonic}, part (ii).
 
 We shall prove the integral $I:=\int_X (Lf)\, d\mu$ which is non-negative (perhaps $+\infty$) must in fact be zero.
 By the Furstenberg Random Ergodic Theorem the functions
 $$
 A_N(x,\omega) := \frac 1 N \sum_{k=0}^{N-1} Lf(g_k(\omega)\cdots g_1(\omega) x)
 $$
 converge $(x,\omega)$-a.e. to $I$ (for $k=0$, we take $Lf(x)$ in the sum).
 Note also that $A_N(x,\omega)$ is always non-negative.
 
 We now rewrite the expression for $A_N$ using the definition of $Lf$:
 \begin{align*}
 A_N(x,\omega)&=\frac 1 N \sum_{k=0}^{N-1}\int_G \left[ f(g g_k(\omega)\cdots g_1(\omega) x)\, d\nu(g) -\right.\\
	      & \hskip 0.8in -\left.f(g_k(\omega)\cdots g_1(\omega) x)\right]\, d\nu(g)
 \end{align*}
 Taking expectations over $\omega$ we find a telescoping sum
 $$
 A_N(x)= \frac 1 N \left[\int_{G^{N}} f(g_{N}\cdots g_1 x) \, d\nu^{\otimes N} - f(x)\right]
 $$
 Now consider the functions
 $$
 B_N(x,\omega):= \frac 1 N \left[f(g_{N+1}(\omega)\cdots g_1(\omega) x) - f(x)\right]
 $$ 
 Because $f$ is tame, for fixed $x$ this function is bounded.
 But $f$ is also of sublinear growth, so $\omega$-pointwise this function goes to zero as $N\to \infty$.  
 From the Dominated Convergence Theorem, we conclude that for a.e. $x$ the $\omega$-integral of $B_N$ must converge to zero as $N\to \infty$.
 
 Note that we have
 $$
 A_N(x) = \bE[B_N(x,\omega)]
 $$
 Moreover, the above quantity is non-negative.
 From the convergence of the integral of $B_N$ to zero, we conclude that $A_N(x)$ also converges pointwise to zero. 
 This implies that the integral $I:=\int_X Lf\, d\mu$ must also be zero.
 
 For Part (ii), fix $A\in \bR$ and consider the function
 $$
 f_A := A+\max(-A, f)
 $$
 This is still subharmonic, since the maximum of two subharmonic functions is subharmonic. 
 But it satisfies the assumptions of part (i) and is thus constant.
 Sending $A$ to $+\infty$, we conclude $f$ itself must be constant. 
\end{proof}

\section{Semisimplicity}
\label{sec:ssimplicity}

In this section we consider an ergodic $\SL_2\bR$-invariant measure $\mu$ on some stratum $\cH$.
First we consider an integrable cocycle $H_\bC$ over $\mu$ which gives a variation of Hodge structures on every \Teichmuller disk.
In particular, it is invariant under $K=SO(2)$.

Moreover, we make the boundedness assumption on the cocycle matrix $A(g,x)$ for $g\in \SL_2\bR, x\in X$ 
\begin{align}
\label{eqn:bdd_cocycle}
\log \norm{A(g,x)}\leq C\norm {g}
\end{align}

This is satisfied in the case of the Kontsevich-Zorich cocycle and cocycles obtained from it by tensor operations.
This was first proved by Forni in \cite{Forni} but see also \cite[Lemma 2.3]{FMZ}

Using the results on random walks and curvature of Hodge bundles, we prove the Theorem of the Fixed Part. 
It states that a section of this cocycle flat along every \Teichmuller disk must have each $(p,q)$ component flat as well.

This theorem applies to endomorphism bundles of the Kontsevich-Zorich cocycle (i.e. the Hodge bundle) or tensor powers thereof.
To use this, we need the reductivity of the algebraic hull (see Remark \ref{remark:reductive}).

The above discussion gives a semisimplicity theorem similar to De\-ligne's in the case of usual variations of Hodge structure.

\subsection{Theorem of the Fixed Part}

We need some preliminary results.

\begin{lemma}
\label{lemma:section_central}
 Let $g_t$ be an ergodic measure-preserving flow on a space $(X,\mu)$ and let $H$ be some integrable linear cocycle over the flow.
 Suppose that $\phi$ is a measurable section of the cocycle which is invariant under the flow (we assume some underlying linear representation).
 Then $\phi$ must a.e. lie in the central Lyapunov subspace (i.e. it has Lyapunov exponent zero).
\end{lemma}
\begin{proof}
 If not, then $\norm{\phi}$ would grow along a.e. trajectory.
 But the flow recurs to sets where the norm of $\phi$ is bounded.
\end{proof}

\begin{lemma}
\label{lemma:delbar_subharm_const}
 Suppose $f$ is a $\mu$-measurable function, invariant under $K:=\SO(2)$.
 It descends to \Teichmuller disks, and assume it is subharmonic on $\mu$-almost all of them.
 This means that $\del\delbar f\geq 0$ in the sense of Section \ref{subsec:positivity} (note the change from $\delbar \del$ to $\del \delbar$).
 
 Denote by $f^+:=\max(0,f)$ the positive part of $f$, also subharmonic.
 Suppose that $f^+$ grows sublinearly along a.e. \Teichmuller geodesic (non-uniformly in the geodesic).
 Finally, suppose that $|f^+(x)-f^+(gx)|\leq C \norm{g}$ for some fixed $C$ and for every $g\in \SL_2\bR$.
 
 Then $f$ must be $\mu$-a.e. constant. 
\end{lemma}
\begin{proof}
 Pick a $K$-bi-invariant measure $\nu$ on $\SL_2\bR$, with compact support which generates the group.
 Then $f$ is also $\nu$-subharmonic in the sense of Definition \ref{def:subharm}.
 
 As a consequence of the Oseledets theorem, random walk trajectories track geodesics with sublinear error.
 This means we have a rate of drift $\delta>0$ and a random geodesic $\gamma_\bullet(\omega)$ on $K\backslash \SL_2\bR$ such that
 $$
 d\Big(\left[ g_n(\omega)\cdots g_1 (\omega) x \right], [\gamma_{\delta n}(\omega) x]\Big) = o(n)
 $$
 Here $\left[-\right]$ denotes the projection or equivalence class in $K\backslash \cH$, i.e. the stratum divided by the action of $K$.
 
 Note that the function $f$ satisfies the assumptions of Proposition \ref{prop:subharm_const} part (ii).
 Tameness and subharmonicity are part of the current assumptions and the sublinear tracking of \Teichmuller geodesics gives the sublinear growth along paths of the random walk.
 We conclude $f$ must be a.e. constant.
\end{proof}

\begin{theorem}[Theorem of the Fixed Part]
\label{thm:fixed_part}
 Let $H_\bC$ be a variation of Hodge structures satisfying the boundedness assumption from the beginning of the section (see equation \eqref{eqn:bdd_cocycle}).
 Suppose that $\phi$ is a measurable section of $H_\bC$, flat along a.e. $\SL_2\bR$-orbit.  
 
 Then each $(p,q)$-component of $\phi$ is also flat along a.e. $\SL_2\bR$-orbit.
\end{theorem}
\begin{proof}
 Write $\phi=\phi^{w,0}+ \cdots + \phi^{p,w-p}$.
 Recall (equation \ref{eqn:GM_vs_Hg}) that we have the relation
 $$
 \nabla^{GM} = \nabla^{Hg} + \sigma_\bullet + \sigma^\dag_\bullet
 $$
 Because $\phi$ is flat, by inspecting the $(p-1,w-p+1)$ component of $\nabla^{GM}\phi$ we see that $\sigma_p\phi^{p,w-p}=0$ (along \Teichmuller disks).
 
 Consider now the projection of $\phi$ to the bundle $\cH^{p,w-p}$.
 Since $\cH^{p,w-p}:=\cF^p/\cF^{p+1}$ is a quotient of holomorphic bundles and $\phi$ is a holomorphic section of $\cF^{p}$, the projection is also holomorphic. 
 It also equals $\phi^{p,w-p}$ and we denote it by $\psi$ for simplicity.
 
 Applying Lemma \ref{lemma:delbar_del_log} (note the switch in order of $\delbar$ and $\del$), we find
 $$
 \del\delbar \log \norm{\psi}^2 = -\frac{\ip{\Omega\psi}{\psi}}{\norm{\psi}^2} + \frac{\norm{\psi}^2\cdot \ip{\nabla \psi}{\nabla\psi} - \ip{\nabla{\psi}}{\psi}\cdot\ip{\psi}{\nabla\psi}}{\norm{\psi}^4}
 $$
 Note that the second term is positive by Cauchy-Schwartz.
 For the first one, recall that by Proposition \ref{prop:curvature_Hg}
 $$
 \Omega_{\cH^{p,w-p}}=\sigma^\dag_p\wedge\sigma_p + \sigma_{p+1}\wedge \sigma^\dag_{p+1}
 $$
 Because $\sigma_p \psi=0$, we apply Proposition \ref{prop:curvature_Hg_sect} to find
 $$
 \ip{\Omega_{\cH^{p,w-p}}\psi}{\psi}  = \ip{\sigma_{p+1}^\dag \psi}{\sigma_{p+1}^\dag \psi}
 $$
 The term above is negative, so the function is subharmonic along a.e. \Teichmuller disk.
 Note that we might have $\delta$-masses coming from zeroes of $\psi$, but the function will stay subharmonic (see Remark \ref{remark:delta_masses} (ii)). 
 
 By Lemma \ref{lemma:section_central} we have that the positive part of $\log \norm{\phi}^2$ grows sublinearly along a.e. \Teichmuller geodesic. 
 The same must be true of the positive part of each of its $(p,q)$-components, in particular of $\log \norm{\psi}^2$.
  
 We can thus apply Proposition \ref{lemma:delbar_subharm_const} to conclude that $\log \norm{\psi}^2$ must be constant along a.e. \Teichmuller disk. 
 
 Looking at $0=\delbar\del \norm{\psi}^2$ (see Remark \ref{remark:delbardel_norm}) we find that $\sigma_{p+1}^\dag\psi=0$ and $\nabla^{Hg}\psi=0$. 
 By looking at the relationship between the Gauss-Manin and Hodge connections (see equation \eqref{eqn:GM_vs_Hg}) we conclude $\psi$ is flat for the Gauss-Manin connection.
 
 Subtracting $\psi=\phi^{p,w-p}$ from $\phi$, the above argument can be iterated.
\end{proof}

We also record for future use the next result.
Note that it is also used in \cite{EM} to compare the volume forms coming from the symplectic pairing and the Hodge norm.

\begin{corollary}
\label{cor:for_Eskin}
Suppose $H$ is a cocycle which induces a variation of Hodge structure on \Teichmuller disks with appropriate boundedness conditions (e.g. a tensor power of the Kontsevich-Zorich cocycle).
Suppose that $\phi$ is a measurable global section of $H$, flat along a.e. $\SL_2\bR$-orbit.

Then the Hodge norm of $\phi$ is a.e. constant, and each $(p,q)$-component of $\phi$ is flat along a.e. $\SL_2\bR$-orbit.
Moreover, each $(p,q)$-component also has constant Hodge norm.
\end{corollary}

\subsection{Deligne semisimplicity}

\paragraph{Setup} In this section we denote by $E$ the Hodge bundle or some tensor power, defined over an $\SL_2\bR$-invariant measure $\mu$ in some stratum.
For the real and complex bundles, we use the notation $E_\bR$ and $E_\bC$.
%Let also $F\subset E$ denote the Forni subspace.
%A definition is given in the paper of Avila, Eskin and M\"oller \cite{AEM} where some of its properties are developed.

\begin{theorem}
\label{thm:v_cv}
 Suppose that $V\subset E_\bC$ is an $\SL_2\bR$-invariant subbundle.
 Then $C\cdot V$ is also $\SL_2\bR$-invariant, where $C$ is the Hodge-star operator (in higher weight, the Weil operator).
\end{theorem}

\begin{corollary}
 Suppose $V\subset E_\bR$ is $\SL_2\bR$-invariant.
 Then so is $C\cdot V$.
\end{corollary}
\begin{proof}[Proof of Corollary]
 Apply the previous theorem to $V_\bC:=V\otimes \bC$ and note that $C$ is an operator defined over $\bR$.
\end{proof}

\begin{proof}[Proof of Theorem \ref{thm:v_cv}]
 By Remark \ref{remark:reductive}, we know that the cocycle corresponding to $E$ has reductive algebraic hull.
 In particular, any $\SL_2\bR$-invariant subbundle has an invariant complement. 
 Denote this complement by $V^\perp$ and let $\pi_V\in \End(E_\bC)$ be the projection to $V$ along $V^\perp$.
 This projection operator is $\SL_2\bR$-invariant, because the bundles are.
 
 We apply the Theorem of the Fixed Part \ref{thm:fixed_part} to conclude that $C\cdot \pi_V$ is also $\SL_2\bR$-invariant.
 But this last operator is projection to $C\cdot V$ along $C\cdot V^\perp$, so we conclude $C\cdot V$ must be $\SL_2\bR$-invariant.
\end{proof}

\begin{remark}
\label{remark:hg_struct}
 It is clear that Theorem \ref{thm:v_cv} is valid if we replace the Hodge star operator by any other element of the Deligne torus $\bS$.
 This is relevant in the case of higher-weight variations.
\end{remark}

The proof of the following result is along the lines presented by Deligne in \cite{Deligne}.

\begin{theorem}[Deligne semisimplicity]
\label{thm:ssimple}
There exist $\SL_2\bR$-invariant bundles $V_i\subset E$ and vector spaces $W_i$ equipped with Hodge structures and an isomorphism
 $$
 E\cong \bigoplus_i V_i\otimes_{A_i} W_i 
 $$
 Moreover, each $V_i$ carries a variation of Hodge structure making the above isomorphism compatible.
 The $A_i$ are division algebras which act on $V_i$, compatible with Hodge structures. 
 They also act compatibly on $W_i$ (see Remark \ref{remark:isotypical} for a discussion of these conditions).
 
 Any $\SL_2\bR$-invariant bundle $V'\subset E$ is of the form
 $$
 V' = \bigoplus_i V_i\otimes_{A_i} W_i'
 $$
 where $W_i'\subset W_i$ are $A_i$-submodules.
 In the case of complexified bundles $E_\bC$ a Hodge structure is understood as defined in Theorem \ref{thm:Deligne_ss_intro}.
\end{theorem}

\begin{remark}
 When $E$ is the Hodge bundle, we know that every invariant subbundle is either symplectic or inside the Forni bundle (see \cite{AEM}).
This means that in the decomposition above, besides the Forni subspace, the only other possibility is to have $W_i$ a vector space with positive-definite inner product (i.e. a polarized Hodge structure of weight $0$) and $V_i$ some weight $1$ polarized variation of Hodge structure.
\end{remark}

\begin{proof} 
 Let $V$ be an $\SL_2\bR$-invariant subbundle of minimal dimension.
 Because it is of minimal dimension its endomorphism algebra, denoted $A$, is a division algebra.
 Let $W$ denote the space of morphisms of $\SL_2\bR$-invariant bundles from $V$ to $E$ (not required to respect the Hodge structures).
 $$
 W:=\Hom_{\SL_2\bR}(V,E)
 $$
 Because $A$ acts by endomorphisms on $V$, it also acts on the left on $W$ by precomposition.
 We can then define the natural evaluation map
 $$
 ev:V\otimes_A W \to E
 $$
 Denote its image by $E'$.
 By Theorem \ref{thm:v_cv} (see also Remark \ref{remark:hg_struct}) it follows that $E'$ is a sub-variation of Hodge structure.
 By considering the orthogonal to $E'$ (the definite or indefinite metric give the same complement) we reduce to applying the argument below by induction to this complement.
 
 Any $\phi\in W$ is either injective or zero because the dimension of $V$ is smallest possible.
 Therefore, we have an isomorphism $V\otimes_A W\to E'$.
 Because $V$ has no invariant subbundles, we have
 $$
 \End_{\SL_2\bR}(E') \cong \End_{\SL_2\bR}(V\otimes_A W) \cong \End_A(W)
 $$
 Note that the first object has a natural Hodge structure inherited from the underlying variation, thus it induces one on $\End_A(W)$. 
 Lemma \ref{lemma:Hg_end} (see below) provides $W$ with a Hodge structure. 
  
 We want to endow $V$ with a Hodge structure such that the isomorphism $V\otimes_A W \to E'$ is compatible.
 The bundle $V$ is naturally identified with the subbundle of $\Hom(W,E')$ which is equivariant for the action of the algebra $\End_A(W)$ (acting by $\End(E')$ on the second factor).
Namely, every $v\in V$ gives an evaluation map $W\to E'$ (recall that $W$ itself is a $\Hom$-space).
The subbundle thus-obtained is characterized by the equivariance property for the action of $\End_A(W)$.
 
 This provides $V$ with the required Hodge structure.
 Note that the structures on $V$ and $W$ are unique up to a simultaneous shift (in opposite directions).
 
 The proof of the first part is now complete, as we have endowed the required spaces with Hodge structures.
 
Consider now a general invariant subbundle $V'\subset E$ and the given direct sum decomposition
$$
E=\oplus V_i\otimes_{A_i} W_i
$$
Let $\pi_i$ be the projection onto the factor with index $i$.
We claim $V'=\pi_1 V' \oplus (1-\pi_1)V'$.
If this is proved, then we can iterate the argument to $(1-\pi_1)V'$.
It is also clear that any invariant subbundle of $V_i\otimes_{A_i} W_i$ has to be of the form $V_i\otimes_{A_i} W_i'$ for some $A_i$-submodule $W_i'\subset W_i$.

To prove the claimed decomposition of $V'$, suppose that $\ker \pi_1$ and $\ker (1-\pi_1)$ don't span $V'$.
Their span has some non-trivial invariant complement $V''$.
But the image of $V''$ under $\pi_1$ and $1-\pi_1$ is isomorphic to $V''$ (and $\SL_2\bR$-invariant) and reversing one of the arrows, we get an embedding of $V_1$ into $E$ which was not accounted for by $W_1$.
This is a contradiction.
\end{proof}

\begin{remark}
\label{remark:isotypical}
 The above proof (and statement) applies to both the real and complex Hodge bundles, so this specification is omitted from the notation.
 The spaces $W_i$ correspond to possible isotypical components of the bundles, but formulated in an invariant way.
 
 The algebras $A_i$ can only arise over $\bR$, since there are no division algebras over $\bC$.
 In this case, they take into account possible symmetries of the real decomposition.
 The division algebra $\bC$ arises when after extension of scalars from $\bR$ to $\bC$ we have a further splitting.
 The quaternions $\bH$ arise when this splitting has futher symmetries, see for example \cite{MYZ}.
 Typically however, the algebras $A_i$ are just the scalars, and the spaces $W_i$ are one-dimensional, so the summands in Theorem \ref{thm:ssimple} are just the $V_i$.
 
 The above proof is also compatible with the underlying polarizations, since all the constructions were natural.
 The subvariation $E'$ carries a polarization and it extends to $\End(W)$ and then lifts to $W$.
 It automatically gives one on $V$ by construction.
\end{remark}

We now return to a claim used in the proof.

\begin{lemma}
\label{lemma:Hg_end}
Let $W$ be a finite-dimensional real vector space.
 A Hodge structure on $\End(W)$ which is compatible with the algebra structure comes from a Hodge structure on $W$, unique up to shift of weight.
 
 A similar statement holds for complex vector spaces, where a Hodge structure means a bi-grading.
 The Deligne torus $\bS$ in the proof below is replaced by $\bG_m^2$ (two copies of the multiplicative group).
 
 Finally, the same statement is true for $W$ a module over a division algebra $A$.
 Namely, a Hodge structure on $\End_A(W)$ compatible with the algebra structure gives one on $W$, compatible with the $A$-action.
\end{lemma}
\begin{proof}
 Because $\End(W)$ is a simple algebra, its automorphisms group is $\PGL(W)$.
 To give a Hodge structure on a space is the same as to give an action of the Deligne torus $\bS$ on it (see \cite[Definition 1.4]{Deligne_travaux}).
 Thus, we have a homomorphism $\bS\to \PGL(W)$.
 
 Consider the exact sequence
 $$
 1\to \bG_m\to \GL(W) \to \PGL(W) \to 1
 $$
 We see that a homomorphism of $\bS$ to $\PGL(W)$ must lift to $\GL(W)$ because any extension of $\bS$ by $\bG_m$ splits (non-uniquely).
 This lift gives $W$ a Hodge structure.
 
 When a division algebra $A$ acts on $W$, the same argument applies but with $\GL(W)$ replaced by $GL_A(W)$, and $\PGL(W)$ by $\PGL_A(W)$.
 
 To prove that a splitting always exists, suppose given an exact sequence of real algebraic groups
 $$
 1\to \bG_m \to G \to \bS \to 1
 $$
 Consider the dual exact sequence of character lattices
 $$
 0 \leftarrow \bZ \leftarrow \bZ^3 \leftarrow \bZ^2 \leftarrow 0
 $$
 Each group has an action of $\bZ/2=\left<\sigma\vert\sigma^2=1\right>$.
 On $\bZ$ it is trivial, on $\bZ^2$ it is by $\sigma(x,y)=(y,x)$.
 By looking at the matrix of $\sigma$ on $\bZ^3$, it can be seen directly that the sequence must split as a sequence of $\bZ/2$-modules.
\end{proof}

\section{Rigidity}
\label{sec:rigidity}

We restrict attention to an affine invariant manifold $\cM$ as defined in Section \ref{subsec:background}.
Fix some tensor power of the Hodge bundle and call it $E$.
Using that invariant subbundles of $E$ are Hodge-orthogonal, in this section we show that these subbundles must vary analytically (even polynomially) on the base.

The first step is to show that invariant bundles vary real-analytically on a.e. stable or unstable leaf.
This follows from Hodge orthogonality, combined with flatness of the Lyapunov filtration.

Next, from real-analyticity one can deduce in fact polynomiality along leaves.
This follows from the contraction properties of the \Teichmuller geodesic flow.

Finally, the results are assembled to show that the subbundles must vary locally polynomially. 

\subsection{A basic example}

To illustrate the technique, consider the situation when all Lyapunov exponents are distinct:
$$
1=\lambda_1>\lambda_2 > \cdots > \lambda_g > 0 > -\lambda_g >\cdots > -\lambda_1 = -1
$$
Using Theorem \ref{thm:ssimple}, we see that the decomposition of the Hodge bundle $E$ must be of the form
$$
E= E_1\oplus\cdots \oplus E_k
$$
where each $E_i$ is irreducible and $\SL_2\bR$-invariant.
Each Lyapunov subspace $E^{\lambda_i}$ is one-dimensional and must belong to one of the bundles $E_{j(\lambda_i)}$.
Consecutive exponents which belong to the same $E_i$ can be grouped into blocks.
We thus have a decomposition
$$
E = B_1 \oplus\cdots B_{2b+1}
$$
Here we have
$$
B_l = E^{\lambda_k}\oplus E^{\lambda_{k+1}}\oplus\cdots\oplus E^{\lambda_{k+n}}
$$
where $j(\lambda_{k-1})\neq j(\lambda_k)=\cdots= j(\lambda_{k+n})\neq j(\lambda_{k+n+1})$.
There is an odd number of blocks because there is a middle one and the rest appear in symmetric pairs.

On each $E_i$ we have an induced filtration by its blocks $E_{i,1}\subsetneq E_{i,2}\cdots \subsetneq E_i$ with
$$
E_{i,j} = B_{i_1}\oplus\cdots\oplus B_{i_{n_j}}
$$
\begin{proposition}
\label{prop:simple_spec}
 For a.e. $x\in \cM$ the subspaces $E_{i,j}$ inducing the filtration agree a.e. on $\cF^+(x)$ with a real-analytic family.
 Here, $\cF^+(x)$ is the unstable foliation.
In local period coordinates where $x\mapsto (\Re x, \Im x)$ this is
$$
\cF^+(x)=\cM\cap\{(\Re x + u, \Im x) \vert u\in H^1_{rel}(\bR)\}
$$
 
 In particular, the bundle $E_i$ itself varies real-analytically on a.e. leaf.
\end{proposition}
\begin{remark}
 There is a corresponding statement for the opposite filtration and the stable direction.
\end{remark}

\begin{proof}
 Let us denote by $F_{\leq \lambda_i}$ the Lyapunov filtration of the entire Hodge bundle $E$, which contains all the Lyapunov subspaces with exponent at most $\lambda_i$.
 This filtration is flat along the unstable leaf.
 
 Choose $\lambda_j$ at the right-most edge of one of the blocks $B_l$ so as to have the decomposition
 $$
 F_{\leq \lambda_j} = \bigoplus_i E_{i,j_i}
 $$
 
 We can now argue by induction on the eigenvalue $\lambda_j$.
 The very first block always corresponds to the first coordinate in the $\SL_2\bR$-action, so the claim is clear.
 Now, in the above decomposition of $F_{\leq \lambda_j}$, all but the last term $E_{i,j_i}$ which contains $\lambda_j$ are known to vary real-analytically by the induction hypothesis.
 
 However, by Theorem \ref{thm:ssimple} the various blocks are Hodge-orthogonal, so that particular $E_{i,j_i}$ is the Hodge-orthogonal of a real-analytically varying family inside a flat subspace.
 We conclude that it must be itself real-analytic.
\end{proof}

\subsection{Leafwise analyticity}

To deal with the situation when Lyapunov exponents have multiplicities, we need the following result stated as Corollary 4.5 in \cite{EM}.

\begin{proposition}
\label{prop:EM_trap_flat}
 Suppose $M$ is a $g_t$-invariant subbundle and $F_{\leq \lambda_k}\subseteq M \subseteq F_{\leq \lambda_{k+1}}$ where as above $F_{\leq \lambda_\bullet}$ is the (increasing) Lyapunov filtration and $\lambda_k>\lambda_{k+1}$.
 Then $M$ is locally constant along the unstable leaves. 
\end{proposition}

We can now prove the main result of this section.
It applies to the Hodge bundle over any tensor power of the Kontsevich-Zorich cocycle, still denoted by $E$.

\begin{theorem}
 In the decomposition provided by Theorem \ref{thm:ssimple}
 $$
 E=\left(\bigoplus_i E_i \otimes D_i \right)
 $$
 each bundle $E_i\otimes D_i$ varies real-analytically on a.e. stable and a.e. unstable leaf. 
\end{theorem}
\begin{proof}
 Consider the decomposition of each subbundle into Lyapunov subspaces
 $$
 E_i\otimes D_i = B_i^{\lambda_{j_1}}\oplus \cdots \oplus B_i^{\lambda_{j_n}}
 $$
 Just like in Proposition \ref{prop:simple_spec}, we shall argue that the corresponding filtration varies real-analytically on unstable (resp. stable) leaves.
The issue is that the same exponent could occur in different pieces of the Hodge decomposition of $E$.
 
 Consider the Lyapunov filtration of $E$ which is flat along unstable leaves:
 $$
 F_{\leq \lambda_1} \subsetneq F_{\leq \lambda_2} \subsetneq \cdots \subsetneq F_{\leq -\lambda_1} = E
 $$
 Because the individual blocks $B_{i_0}^{\lambda_{k+1}}$ are $g_t$-invariant, we have
 $$
 F_{\leq \lambda_k} \subseteq span(B_{i_0}^{\lambda_{k+1}} + F_{\lambda_k}) \subseteq F_{\lambda_{k+1}}
 $$
 We can thus apply Proposition \ref{prop:EM_trap_flat} to conclude that the middle bundle above is flat along unstable leaves.
 Note that we have the decomposition
 $$
 F_{\leq \lambda_k} = \bigoplus_i \bigoplus_{\lambda_j\leq \lambda_k} B_i^{\lambda_j}
 $$
 We can apply induction and deduce, just like in Proposition \ref{prop:simple_spec}, that $B_{i_0}^{\lambda_1}\oplus \cdots \oplus B_{i_0}^{\lambda_{k+1}}$ is the Hodge-orthogonal of something real-analytic inside something flat.
 It therefore must itself be real-analytic.

The proof for the stable foliation is analogous.
\end{proof}

\subsection{Leafwise to global analyticity}

In the previous section we established analyticity of invariant bundles on a.e. stable or unstable leaf.
In this section we first prove that in fact they must vary polynomially on a.e. leaf.
Combined with a lemma about coordinate-wise polynomial functions, this implies that the bundles vary polynomially on affine manifolds, in particular real-analytically.

\begin{remark}
 We will work below with two cocycles.
 One is a tensor power of the Kontsevich-Zorich cocycle, which is the Gauss-Manin connection on a tensor power of the Hodge bundle $H^1_\bR$.
 By abuse of notation, we continue to denote it by $g_t$ (when considering the \Teichmuller geodesic flow).
 
 The positive part of the Lyapunov spectrum of the Kontsevich-Zorich cocycle is
 $$
 1=\lambda_1>\lambda_2 \geq \lambda_3 \cdots \geq \lambda_g \geq 0
 $$
 The spectral gap inequality $1>\lambda_2$ is due to Forni \cite{Forni} and is key to the argument below.
 
 The second cocycle comes from the action of $g_t$ on the stratum and the induced cocycle on the tangent space.
 We denote this cocycle by $dg_t$.
 The positive part of its Lyapunov spectrum is
 $$
 1+\lambda_1 > 1+\lambda_2 \geq \cdots \geq 1+\lambda_g \geq 1\geq \cdots \geq 1 \geq 1-\lambda_g \geq \cdots \geq 1-\lambda_1
 $$
 Although the last term above is $0=1-\lambda_1$, the above quoted spectral gap result of Forni implies that on unstable leaves, the cocycle $dg_t$ is uniformly expanding at rate at least $1-\lambda_2$.
 The reason is that the subspace corresponding to the zero exponent comes from the centralizer of $g_t$ inside $\GL_2\bR$ and does not appear when restricted to area one surfaces. 
 A discussion of these questions is available in section 5.8 of the survey \cite{Zorich_survey}.
\end{remark}

Let us return to the invariant subbundles which vary real-analytically on a.e. leaf.
To save notation, denote them by $E_i$.
We focus on a fixed one.

\begin{definition}
 For a point $x$ in the affine manifold $\cM$ and vector $v$ in the unstable direction on $\cM$ define the operator
 $$
 \pi(x,v) : E_x \to E_x
 $$
 It is the operator of projection (definite and at the same time indefinite) onto $E_i$ at the point $x+v$, transported to the point $x$ by the Gauss-Manin connection.
\end{definition}

We view $\pi(x,v)$ as a section of the appropriate bundle (obtained from the Hodge bundle), thus $g_t$ acts on it by the Gauss-Manin connection. 
From the equivariance properties of the bundles, we deduce that
\begin{align}
\label{eqn:gt_eqvr}
g_{-t} \pi(x,v) = \pi(g_{-t}x,dg_{-t}v)
\end{align}

Note that the vector $v$ is moved by the cocycle $dg_{-t}$ because it lives in the ambient manifold.
This will be crucial.

\begin{proposition}
\label{prop:leaf_poly}
 Suppose that for a.e. $x$, we have that $\pi(x,v)$ varies real-analytically in $v$, where $v$ is in some small neighborhood of $x$ along the unstable leaf.
 Then for a.e. $x$ we have that $\pi(x,v)$ varies polynomially in $v$. 
\end{proposition}
\begin{proof}
 Equation \eqref{eqn:gt_eqvr} is equivalent to
 \begin{align}
\label{eqn:gt_eqvr'}
 \pi(x,v) = g_t(\pi(g_{-t}x,dg_{-t}v))
 \end{align}
 At a point $x$ where the dependence is real analytic (and the Oseledets theorem holds) we have
 $$
 \pi(x,v) = \sum_{\alpha} c_\alpha(x) v^\alpha
 $$
 Here $\alpha$ is a multi-index and $c_\alpha$ are (measurably varying in $x$) endomorphism of $E_x$.
 
 Using the equivariance properties of $\pi$ under the flow given by equation \eqref{eqn:gt_eqvr'} we have two different Taylor expansions
 \begin{align*}
  \pi(x,v)&= \sum_\alpha g_t c_\alpha(g_{-t}x)(dg_{-t}v)^\alpha\\
  &=\sum_{\alpha} c_\alpha(x) v^\alpha
 \end{align*}
 
 Let $\Lambda$ be the largest Lyapunov exponent of $\End(E)$.
 We will show that for $|\alpha|>\frac \Lambda {1-\lambda_2}$ we must have $c_\alpha = 0$.
 To do this, fix a coordinate neighborhood around $x$ and inside it another set $K$ of positive measure on which $c_\alpha$ is bounded above.
 
 Considering times $t$ such that $g_{-t}x\in K$ we have for any $\epsilon>0$ (as $t\to +\infty$)
 $$\norm{g_t c_\alpha(g_{-t}x)}=o(e^{(\Lambda+\epsilon)t})$$
 On the other hand for any $\epsilon_1<1-\lambda_2$
 $$\norm{dg_{-t}v}=o(e^{-t\epsilon_1})$$
 So we have that
 $$
 \norm{ g_t c_\alpha(g_{-t}x)(dg_{-t}v)^\alpha} = o(e^{t(\Lambda+\epsilon - |\alpha|\epsilon_1)})
 $$ 
 We conclude that whenever $|\alpha|>\frac  \Lambda {1-\lambda_2}$, the corresponding terms in the Taylor expansion of $\pi(x,v)$ must vanish. 
\end{proof}

To continue, we record the following observation which goes back at least to Margulis.

\begin{lemma}
\label{lemma:margulis}
%  Suppose $f:\bR^{n_1}\times \bR^{n_2}\to \bR$ is a measurable function such that for a.e. $x_1\in \bR^{n_1}$ we have that $f(x_1,-)$ is a polynomial in $x_2$, and similarly for a.e. $x_2\in \bR^{n_2}$.
%  
%  Then $f$ agrees a.e. with a polynomial in $x_1$ and $x_2$.
 Let $U_i\subset \bR^{n_i}$ with $i=1..2$ be connected open ``boxes'', i.e. of the form product of intervals.
 Let $f:U_1\times U_2 \to \bR$ be a measurable function.
 Assume that for a.e. $x_1\in U_1$, the function $f(x_1,-)$ agrees a.e. with a polynomial in the variable $x_2\in U_2$.
 Assume that the same holds for the two variables swapped.
 
 Then $f$ agrees a.e. with a polynomial in $x_1$ and $x_2$.
\end{lemma}
\begin{proof}
 \textbf{Step 1:} Assume that two polynomials $p_1,p_2:\bR^n\to \bR$ agree on a set $E$ of positive Lebesgue measure.
 Then they coincide.
 
 We show this by induction.
 In dimension $1$, this is immediate.
 
 Consider now on $\bR^n\times \bR$ with coordinates $(x,t)$ the decomposition into polynomials
 $$
 p_i(x,t) = \sum_k c^i_k(x)t^k
 $$
 By Fubini, there is a positive measure set of $x$ such that a positive measure set of $t$ satisfy that $(x,t)\in E$.
 For such $x$, it must be that $c^1_k(x) = c^2_k(x)$.
 By induction, $c^1_k=c^2_k$ as polynomials.
 
 \noindent \textbf{Step 2:} We now proceed by induction.
 In fact, it suffices to check the claim when $U_2\subset \bR$ is an interval.
 Applying iteratively this simpler case by specializing all but one of the coordinates in $\bR^{n_2}$, the general claim follows.
 
 Note that there exist positive measure sets $E_i\subset U_{n_i}$ such that the degrees of the polynomials in the hypothesis on $f$ are bounded (by some $N$).
 We then have a.e.
 \begin{align*}
 f(x_1,x_2) &= \sum_{|\alpha|<N} c_\alpha(x_2) x_1^\alpha\\
 f(x_1,x_2) &= \sum_{n<N} d_n(x_1) x_2^n
 \end{align*}
 where $x_i\in E_i$ and $c_\alpha(x_2), d_n(x_1)$ are measurable.
 Here $\alpha$ denotes a multi-index, while $n$ a positive integer.
 
 By assumption, for a.e. value of $x_2\in E_2$, the two sides agree a.e. in $x_1$.
 We can thus pick $N+1$ distinct values for $x_2$ and solve to find that $d_n(x_1)$ are a.e. equal to polynomials in $x_1$.
 Note that the determinant of the system to solve is of Vandermonde type, so non-zero.
 
 We conclude that $f$ on $E_1\times E_2$ is a polynomial function.
 Enlarge now $E_1$ and $E_2$ but such that the degrees of the polynomials in the assumption stay bounded.
 Then the same argument shows $f$ is polynomial on the larger set.
 By Step 1, it must be the same polynomial.
 Exhausting $\bR^{n_i}$ by such sets we find that $f$ is a.e. equal to a polynomial.
\end{proof}

Combining the above Lemma with Proposition \ref{prop:leaf_poly}, we prove the next result.

\begin{theorem}
\label{thm:msbl_implies_poly_subspace}
 On an affine invariant manifold $\cM$, an $\SL_2\bR$-invariant measurable subbundle of the Hodge bundle (or its tensor powers) must in fact be polynomial in linear coordinates.
 
 Specifically, let the affine coordinates be $(x,y)\in \bR^N\times \bR^N$ and $A(x,y)$ be the area function, quadratic in $x$ and $y$.
 Assign homogeneous degree $1$ to each of $x$ and $y$, and degree $-2$ to $\frac{1}{A(x,y)}$.
 Then the projection operator, in a flat trivialization, is a matrix with entries polynomials of homogeneous degree zero in the variables $x,y, \frac{1}{A(x,y)}$.
\end{theorem}

\begin{proof}
 We first relate the coordinates used in Proposition \ref{prop:leaf_poly} and the affine coordinates.
 If the affine coordinates are $(x_1,\ldots,x_N,y_1,\ldots,y_N)$, then restricting to area $1$ we can drop one $x$-coordinate (say, the last) and use $(x_1,\ldots,x_{N-1},y_1,\ldots,y_N)$.
 Unstable leaves are linear in these coordinates, so the proposition applies and we get polynomiality for fixed $y$-coordinates.
 An analogous construction works by exchanging the stable and unstable coordinates.
 
 Consider now the operator $\pi(x_1,\ldots,x_N,y_1,\ldots,y_N)$ of projection to the bundle.
 Add a dummy variable $A$ and view $\pi$ as a function of $(x,y,A)$, restricted to the quadratic hypersurface of equation $Area(x,y)-A=0$.
 
 Because we can always project to area $1$ surfaces and $\pi$ is invariant by scaling, we find
 $\pi(x,y,A)=\pi\left(\frac {x} {\sqrt{A}}, \frac{y}{\sqrt{A}}, 1\right)$.
 Holding the $y$-coordinates fixed, we see that $\pi$ is a polynomial function in the variables $\frac{x}{\sqrt{A}}$.
 The same is true with $x$ and $y$ swapped.
 
 Applying Lemma \ref{lemma:margulis}, we find that $\pi$ is a polynomial function in the variables $\frac{x}{\sqrt A},\frac{y}{\sqrt A}$.
 Rotation by $180$ degrees in $\SL_2\bR$ changes the signs of $x$ and $y$ and leaves the projection $\pi$ invariant.
 This implies that the polynomial has only terms of even degree in $\frac{x}{\sqrt A}, \frac{y}{\sqrt A}$, in particular we can express it using just $x,y,\frac 1 A$.
 Finally, $\pi$ is invariant by simulatenous scalings of all variables, which implies the polynomial has homogeneous degree zero (assigning degree $-2$ to $\frac 1 A$).
\end{proof}

\section{Applications}
\label{sec:applications}

In this section we collect some applications.
First we consider the algebraic hulls of the Kontsevich-Zorich cocycle.
Then we prove semisimplicity for flat bundles.
Finally, we prove that affine invariant manifolds parametrize Jacobians with non-trivial endomorphisms.

\subsection{Algebraic Hulls}

We show here that the real-analytic and measurable algebraic hulls of the Kontsevich-Zorich cocycle over an affine invariant manifold have to coincide.
For the bundle $E$ (which is the Hodge bundle or a tensor power thereof) we have the associated principal $G$-bundle $P$ of automorphism of the fibers.
In the case of the Hodge bundle, this is a principal $\Sp_{2g}$-bundle.

Given an algebraic subgroup $H\subset G$ to measurably (resp. analytically) reduce the structure group to $H$ is the same as to give an $\SL_2\bR$-equivariant measurable (resp. real-analytic) section $\sigma$ of the bundle $P/H$ (whose fiber is $G/H$).

\begin{theorem}
Given a measurable section $\sigma:X\to P/H$ as above, in local affine coordinates on $X$ it must agree a.e. with a real-analytic section.

In fact, if we think of this section as a choice of a conjugate of $H$ inside the fiber of $P$, in affine coordinates the Lie algebra of this conjugate varies polynomially.
\end{theorem}

\begin{proof}
Suppose given an algebraic group $G$ with a faithful linear representation $\rho$.
Then for any algebraic subgroup $H\subset G$ there exists a tensor power $T$ of $\rho$ and a subspace $R\subset T$ (can take it one-dimensional) such that $H$ coincides with the stabilizer of the subspace.
This fact is classical and due to Chevalley.

We can apply this to our situation and find corresponding to $H$ an invariant subbundle in some tensor power.
From Theorem \ref{thm:msbl_implies_poly_subspace} we see that the corresponding subbundle has to vary polynomially in affine coordinates.
The conclusion about its stabilizer follows.
\end{proof}

\subsection{Flat bundles}

Theorems \ref{thm:fixed_part} and \ref{thm:v_cv} refer to $\SL_2\bR$-invariant subbundles.
Flat subbundles are $\SL_2\bR$-invariant, but not necessarily the other way around.
Therefore, the assumptions of these theorems are weaker than their classical analogues, but so are the conclusions.

In this section, we note that the theorems extend to the flat situation as well.
An object is flat if it is locally constant on the affine manifold.
The theorem of the fixed part extends to all tensor power and so does the semisimplicity result.

\begin{theorem}
\label{thm:ssimple_flat}
 Suppose $\cM$ is an affine invariant manifold and let $E$ denote the Hodge bundle (or any tensor power).
 Denote by $C$ the Hodge-star operator (or any other element of the Deligne torus $\bS$).
 
 If $\phi$ is a global flat section of $E$, then so is $C\cdot \phi$.
 
 If $V\subset E$ is a flat subbundle, then so is $C\cdot V$.  
 
 Also, the flat analogue of the decomposition provided by Theorem \ref{thm:ssimple} is valid. 
\end{theorem}

\begin{proof}
 First we prove the theorem of the fixed part.
 Suppose given a flat section $\phi$ over the entire affine manifold $\cM$.
 We apply the same argument as in Theorem \ref{thm:fixed_part}.
 
 Since $\phi$ if flat on all of $\cM$, it is in particular flat along $\SL_2\bR$-orbits.
 We can apply Corollary \ref{cor:for_Eskin} to find that the Hodge norm of $\phi$ is constant.
 
 If we decompose $\phi=\phi^{w,0}+\cdots+\phi^{p,w-p}$ into its Hodge components, the same corollary gives that each component has constant Hodge norm.
 Because $\phi$ is flat on $\cM$ and $\nabla^{GM}$ can be expressed via equation \eqref{eqn:GM_vs_Hg}, we see (by inspecting the $(p,w-p)$-component of $\nabla^{GM}\phi$) that $\sigma_p \phi^{p,w-p}=0$.
 This holds everywhere on $\cM$, not just along $\SL_2\bR$-orbits.
 
 By Remark \ref{remark:delbardel_norm} applied to $\phi^{p,w-p}$ viewed as a holomorphic section of $\cH^{p,q}$ over $\cM$, we see that $\phi^{p,w-p}$ is flat on $\cM$.
 We can now consider $\phi-\phi^{p,w-p}$ and iterate.

 Once the theorem of the fixed part is available, the proof of semisimplicity and invariance of bundles is as before.
 We only sketch the argument (see \cite{Deligne} and \cite{Schmid}).
 
 In the context of $\SL_2\bR$-invariant bundles, we did not have the monodromy available.
 But for flat bundles we do, and we consider some rational invariant subspace $V\subset E$.
 The monodromy preserves a lattice inside $V$, namely $V\cap E_\bZ$.
 Now take a wedge power such that $V$ becomes one-dimensional, so the monodromy acts by $\pm 1$.
 On a double cover the bundle can now be flatly trivialized by a section $\phi$.
 The theorem of the fixed part applies to $\phi$ and we conclude that $C\cdot V$ must be flat.
 
 Once we have semisimplicity of the monodromy representation over $\bQ$, it follows by standard arguments over the field extensions.
 This gives the claimed results. 
% This should no longer be necessary.
%  Consider now a flat subbundle $V\subset E$.
%  Assume first that it is simple, i.e. it has no flat subbundles.
%  
%  We can take its symplectic complement $V^\perp$ and consider the intersection $V\cap V^{\perp}$.
%  This subspace is isotropic and $\SL_2\bR$-invariant, so it must lie in the Forni subspace $F$ (see \cite{AEM} for definitions).
%  According to the results in \cite{AEM}, the Forni subspace is itself flat on the affine manifold (\cite[Prop. 7.1]{AEM}), and also Hodge-star invariant.
%  
%  Because $V$ is simple, it must be either isotropic or its symplectic orthogonal provides a flat complement.
%  If it is isotropic, we find a complement as follows.
%  We saw $V$ must be inside Forni, which is flat and has a flat complement $F^\perp$.
%  Now, by \cite[Lemma 5.2]{AEM} the monodromy acts isometrically on the Forni subspace (note that their $P^+$ coincides with the Gauss-Manin connection on $F$), so the orthogonal complement of $V$ inside $F$, denoted $V'$, gives a local system.
%  A complement of $V$ is now given by $V'\oplus F^\perp$.
%  
%  When $V$ is not simple, we apply the above argument to a minimal (by inclusion) flat subbundle, find a complement, and then iterate the argument. 
%  Once we have flat complements, the first claim in the present theorem applies to the projection operator to $V$ along its complement and we can proceed as in Theorem \ref{thm:v_cv}. 
%  
%  Finally, the semisimplicity argument works as before, but again only for the Hodge bundle and not the tensor powers. 
\end{proof}

\subsection{Real Multiplication}
\label{subsec:real_mult}
The applications in this section answer a question of Alex Wright.

Recall that according to Theorem 1.5 in \cite{Wright}, over an affine invariant manifold $\cM$ we have a decomposition
\begin{align}
\label{eqn:emb_splitting}
H^1_{\bC} = \left(\bigoplus_{\iota\in I}\bV_\iota \right)\bigoplus \bW
\end{align}
This is a decomposition into pairwise non-isomorphic local systems on $\cM$.
Moreover the $\bV_\iota$ have no local subsystems.

The results of the same paper associate to the affine manifold $\cM$ its field of (affine) definition $k(\cM)$.
This is the minimal field such that in affine coordinates, the affine manifold is defined by linear equations with coefficients in that field.

The summation in equation \eqref{eqn:emb_splitting} is over the set $I$ of all complex embeddings of the field $k(\cM)$.
We also have a distinguished real embedding $\iota_0$ because $k(\cM)$ can be viewed as the trace field of a representation (see \cite[Theorem 1.5]{Wright}).
With these preliminaries, we can now state the main result of this section.

\begin{theorem}
 An affine invariant manifold $\cM$ parametrizes Riemann surfaces whose Jacobians have real multiplication by its field of (affine) definition $k(\cM)$.
 In particular, this field is totally real.
 
 Moreover, the $1$-forms giving the flat structure are eigenforms for the action.
\end{theorem}

\begin{proof}
 We combine the decompositions from equation \eqref{eqn:emb_splitting} and Theorem \ref{thm:ssimple_flat}.
 This implies that each summand $\bV_\iota$ underlies a variation of Hodge structure and in particular is Hodge-star invariant.
 
 Let an element $a\in k(\cM)$ act on the Hodge bundle according to the decomposition from equation \eqref{eqn:emb_splitting}
 $$
 \rho(a):=\left(\bigoplus_{\iota\in I}\iota(a) \right)\oplus 0
 $$
 So $a$ acts by the scalar $\iota(a)$ on the summand corresponding to the embedding $\iota$, and by zero on the remaining part.
 
 Note that this action is compatible with the Galois action on the local system.
 If $\sigma\in Gal(\overline{\bQ}/\bQ)$ then $\sigma(\rho(a))=\rho(a)$.
 
 Because the individual local systems underlie Hodge structures, we see that $\rho(a)$ is a rational endomorphism of type $(0,0)$.
 This gives the desired action of the field $k(\cM)$.
 
 It is known that a field in the endomorphism ring of an abelian variety must be CM or totally real (see \cite[Theorem 5.5.3]{BL_CAV}).
 Because $k(\cM)$ has one real embedding, the latter possibility must occur.
 
 Finally, the $1$-forms giving the flat structure belong to the space $H^1_{\iota_0}$ and are thus eigenforms for real multiplication.
\end{proof}

%-------------End of Main Article-----------
%-------------------------------------------
%-------------------------------------------
%-----------Appendix------------------------
\appendix

\section{Connection to Schmid's work}

In this appendix, we explain the connection of our methods to Schmid's work \cite{Schmid}.
Namely, we show that methods from ergodic theory yield some of the global consequences of his results.

\paragraph{Setup.}
Consider a variation of Hodge structures $E$ over a smooth quasi-projective base $B$.
We do not assume $B$ is compact.
Since it is quasi-projective, through every point $b\in B$ there is at least one Riemann surface of finite type contained in $B$.
In this setup, Theorems 7.22 through 7.25 from \cite{Schmid} hold.
The main one, which implies the rest, is the Theorem of the Fixed Part - Theorem 7.22 in  loc. cit.
We explain how to prove it using ergodic theory.

\begin{theorem}[Theorem of the Fixed Part]
 With notation as above, if $\phi$ is a flat global section over $B$ of $E_\bC$, then each $(p,q)$-component of $\phi$ is also flat.
\end{theorem}

\paragraph{The case of curves.}
We first assume that $B$ is one-dimensional, i.e. a compact Riemann surface with finitely many punctures.
Recall that the universal cover of $B$ maps to the classifying space of the variation.
The appropriate version of the Schwartz lemma implies that if $E$ is non-trivial, then $B$ is necessarily a hyperbolic Riemann surface, of finite area by assumption.
Moreover, the classifying map is a contraction for the appropriate metrics.

We can now consider the geodesic flow on $B$ for the hyperbolic metric and the induced cocycle from the local system of $E$.
Because the classifying map is a contraction, the boundedness assumption in the Oseledets theorem is satisfied.
Moreover, we have not only the geodesic flow but also the full $\SL_2\bR$-action.
The claims from Sections \ref{sec:random_walks} and \ref{sec:ssimplicity} therefore apply verbatim.

This proves the claim in the case when the base is a finite-type Riemann surface.

\paragraph{The general case}
To prove the general case, note we assumed the base is quasi-projective.
In particular, it has lots of $1$-dimensional subvarieties, to which the previous step applies.
The main step in proving the theorem of the fixed part is showing that the subharmonic function coming from the norm of the section must be constant.

In the previous step, we established this when the base is $1$-di\-men\-sio\-nal.
Since $B$ has such $1$-dimensional subvarieties through every point, the subharmonic function is constant on all of them, so on all of $B$.
The proof then proceeds as in Section \ref{sec:ssimplicity}, or \cite[Section 7]{Schmid}.

\section{The Kontsevich-Forni formula}

In this appendix, we provide a derivation of the Kontsevich-Forni formula.
This was first stated by Kontsevich in \cite{Kontsevich}, then proved by Forni in \cite{Forni} (see also \cite{FMZ}).
This appendix contains a proof in the formalism used in this paper.

\paragraph{Setup.} 
Consider some complex manifold $B$ of unspecified dimension, and consider over it a variation of weight-1 Hodge structure $H^1$.
We have the decomposition $H^1_\bC=H^{1,0}\oplus H^{0,1}$.
Inside we have the real (flat) subbundle $H^1_\bR$ with elements of the form $\conj{\alpha}\oplus \alpha$.
We have the positive-definite Hodge norm, and all statements below are with respect to it.
In a change of convention from section \ref{sec:diff_geom}, we take adjoints for the positive-definite metric now.

\begin{proposition}
\label{prop:Kontsevich-Forni}
 Suppose $c_1,\ldots,c_k\in H^1_\bR$ is a basis, at some point of $B$, of an isotropic subspace of $H^1_\bR$.
 Extend $c_i$ using the Gauss-Manin connection to flat sections in a neighborhood.
 Denoting the second fundamental form of the Hodge bundle
 $$\sigma:H^{1,0}\to \Omega^1\otimes H^{0,1}$$
 we have the formula (notation explained below)
 \begin{align}
 \label{eqn:Kontsevich-Forni}
  \del\delbar \log \norm{\Wedge_{i=1}^k c_i}^2=\tr(\sigma\wedge \sigma^\dag)-\tr\left(\sigma\wedge \pi_{\conj{\cC}^\perp}\sigma^\dag\pi_{\cC^\perp}\right)
 \end{align}
 We view $c_i$ as flat sections of $H^1_\bC$ and project them to $H^{0,1}$ to get holomorphic sections $\phi_i$.
 Then the $\phi_i$ span a $k$-dimensional subbundle which we denote $\cC$, and $\cC^\perp$ is its Hodge-orthogonal inside $H^{0,1}$.
 The operator $\pi_{\cC^\perp}$ is orthogonal projection to the space $\cC^\perp$ and $\pi_{\conj{\cC}^\perp}$ is orthogonal projection to its complex-conjugate.
\end{proposition}
\begin{remark}
\begin{enumerate}
 \item [(i)]
 The equation \ref{eqn:Kontsevich-Forni} is an equality of $(1,1)$-forms.
 The left-hand side can be interpreted as a Laplacian once a metric is introduced on $B$.
 For example, the hyperbolic metric on \Teichmuller disks recovers the usual Kontsevich-Forni formula.
 
 \item [(ii)] The right-hand side of the formula is always a non-negative  $(1,1)$-form.
 This is because adjoints are for a \emph{positive-definite} hermitian inner-product.
\end{enumerate}
\end{remark}

\paragraph{Notation.} Write the $c_i$ in their Hodge decomposition
$$
c_i=\conj{\phi_i}\oplus \phi_i \hskip 0.2cm \textrm{ where } \phi_i \textrm{ holomorphic section of } H^{0,1}=H^1_\bC/H^{1,0}
$$
\begin{proposition}
 The isotropy condition on $c_i$ gives the pointwise on $B$ equality of Hodge norms
 $$
 \norm{\Wedge_{1}^k c_i} =2^k\norm{\Wedge_{1}^k \phi_i}
 $$
\end{proposition}
\begin{proof}
 If we apply a fixed real $k\times k$ matrix to the $c_i$ everywhere on $B$, then the claimed equality is not affected - both sides are rescaled by the determinant of the matrix.
 To check the equality at some given point of $B$, we can choose a real linear change of variables for the $c_i$ such that at the considered point, the $c_i$ are also Hodge-orthogonal.
 
 Combined with the isotropy condition on $c_i$ we find that the $\phi_i$ must also be Hodge-orthogonal.
 Indeed, the $c_i$ being Hodge-orthogonal implies the real part of $\ip{\phi_i}{\phi_j}$ has to vanish.
 The isotropy condition implies vanishing of the imaginary part.
 
 But in this situation, the formula can be checked directly.
 Therefore, the asserted equality holds everywhere.
\end{proof}

\begin{proof}[Proof of Proposition \ref{prop:Kontsevich-Forni}]
By the previous result, we need to compute $\delbar \del\log\norm{\wedge \phi_i}^2$.
Recall $\Wedge_{i=1}^k \phi_i$ is a holomorhpic section of $\Wedge^k H^{0,1}$ and so we shall use Lemma \ref{lemma:delbar_del_log} to compute the desired expression.

Recall from Section \ref{sec:diff_geom} the relation between the Gauss-Manin and Hodge connections on $H^1$
$$
\nabla^{GM}=\nabla^{Hg}+\sigma-\sigma^\dag
$$
Because the $c_i$ are flat for $\nabla^{GM}$, looking at the component in $H^{0,1}$ we find
$$
\nabla^{Hg}\phi_i = -\sigma \conj{\phi_i}
$$
From now on, $\nabla$ denotes $\nabla^{Hg}$ and we focus on the bundle $H^{0,1}$.
We shall use the Leibniz rule for the connection and curvature
\begin{align*}
\nabla \Wedge_{i=1}^k \phi_i &= \sum_{i=1}^k \phi_1\wedge\cdots \wedge \nabla \phi_i \wedge\cdots\wedge \phi_k\\
\Omega_{\wedge^k H^{0,1}} \Wedge_{i=1}^k \phi_i &= \sum_{i=1}^k \phi_1\wedge\cdots \wedge \Omega \phi_i \wedge\cdots\wedge \phi_k
\end{align*}
We abuse notation and denote by $\nabla$ the connection on both $H^{0,1}$ and its wedge powers, but we distinguish the curvatures.
Denoting by $\phi:=\phi_1\wedge\cdots\wedge \phi_k$ Lemma \ref{lemma:delbar_del_log} reduces the proof to evaluating
\begin{align}
\label{eqn:curv_KF}
-\frac{\ip{\Omega_{\wedge^k} \phi}{\phi}}{\norm{\phi}^2} - \frac{\ip{\nabla \phi}{\phi}\cdot\ip{\phi}{\nabla \phi}-\norm{\phi}^2\cdot \ip{\nabla \phi}{\nabla \phi} }{\norm{\phi}^4}
\end{align}
We need to check the pointwise equality of the above $(1,1)$-form and the right-hand side of Proposition \ref{prop:Kontsevich-Forni}.
The minus sign comes from the switched order of $\del$ and $\delbar$ in Lemma \ref{lemma:delbar_del_log} and the proposition we are proving.

Remark that we are proving a pointwise equality.
In particular, if we apply any fixed $k\times k$ complex matrix to the sections $\phi_i$ the value given by \ref{eqn:curv_KF} does not change.
Thus, to prove the claimed equality at a point of $B$ we can apply a matrix to assume that the $\phi_i$ are mutually orthogonal and of unit norm at the considered point.
For the calculation, we also complete them to an orthonormal basis $\{\phi_i\}_{i=1}^g$ of the fiber considered.

Finally, equation \eqref{eqn:curvature_ip} gives $\Omega_{H^{0,1}}=-\sigma\wedge \sigma^\dag$ (recall we are taking adjoints for the positive-definite metric now, hence the minus sign).
Denote the entries of $1$-form valued maps $\sigma$ and $\sigma^\dag$ by
\begin{align*}
\sigma \conj{\phi_i}& = \sum_{j=1}^g \sigma_i^j \phi_j\\
\sigma^\dag \phi_k &= \sum_{l=1}^g (\sigma^\dag)_{k}^l \conj{\phi}_l
\end{align*}
We then have, using orthonormality of $\{\phi_i\}_{i=1}^g$
\begin{align*}
 -\ip{\Omega_{\wedge^k} \phi}{\phi}&= \ip{\sum_{i=1}^k \phi_1\wedge\cdots\wedge (\sigma\sigma^\dag\phi_i)\wedge\cdots\wedge \phi_k }{ \phi_1 \wedge \cdots\wedge \phi_k}=\\
 &=\sum_{i=1}^k\sum_{j=1}^g \sigma^i_j\wedge \left(\sigma^\dag\right)_i^j
\end{align*}
We next have
\begin{align*}
\ip{\nabla\phi}{\phi}&=\ip{\sum_{i=1}^k \phi_1\wedge\cdots\wedge (-\sigma\conj{\phi}_i)\wedge\cdots\wedge \phi_k}{\phi_1\wedge\cdots\wedge \phi_k}=\\
&=-\sum_{i=1}^k \sigma_i^i
\end{align*}
We then find
$$
\ip{\nabla\phi}{\phi}\cdot \ip{\phi}{\nabla\phi} = \left(\sum_{i=1}^k \sigma_i^i\right)\cdot \left(\sum_{i=1}^k \conj{\sigma}_i^i\right)
$$
We also have
\begin{align}
\label{eqn:nabla_phi}
\ip{\nabla\phi}{\nabla \phi} =& \left\langle\sum_{i=1}^k \phi_1\wedge\cdots \sigma^i_i \phi_i\cdots\wedge \phi_k + \sum_{k<l}\phi_1\wedge\cdots \sigma_i^l \phi_l\cdots\wedge \phi_k,\right.\notag\\
&\left.\sum_{i=1}^k \phi_1\wedge\cdots \sigma^i_i \phi_i\cdots\wedge \phi_k + \sum_{k<l}\phi_1\wedge\cdots \sigma_i^l \phi_l\cdots\wedge \phi_k\right\rangle=\notag \\
&=\left(\sum_{i=1}^k \sigma_i^i\right) \left(\sum_{i=1}^k \conj{\sigma}_i^i\right) + \sum_{i=1}^k \sum_{l=k+1}^g \sigma_i^l \conj{\sigma}_i^l 
\end{align}
Note that $\norm{\phi}=1$ by our normalization and we also have that $\left(\sigma^\dag\right)^j_i=\conj{\sigma}^i_j$.
We can now combine all three terms to get the claimed formula.
The first term in equation \eqref{eqn:nabla_phi} cancels the $\ip{\nabla\phi}{\phi}$ term.
The second term of the same equation provides the needed contribution to the desired formula.
Indeed, summing up all the terms we obtain
$$
\sum_{i=1}^g \sum_{l=1}^g \sigma_i^l \conj{\sigma}_i^l - \sum_{i=k+1}^g \sum_{l=k+1}^g \sigma_i^l \conj{\sigma}_i^l 
$$
and this corresponds to the right-hand side of \eqref{eqn:Kontsevich-Forni} written out explicitly.
\end{proof}
%----------End of Appendix------------------
%-------------------------------------------
%-------------------------------------------
%-----------Bibliography--------------------
%\newpage
\bibliographystyle{sfilip}
\bibliography{deligne_ss_analytic}
%-------End of Bibliography-----------------
\end{document}